\documentclass[12pt, reqno]{amsart}
\usepackage{amsmath, amstext, amsbsy, amssymb, amscd}
\usepackage{amsmath}
\usepackage{amsxtra}
\usepackage{amscd}
\usepackage{amsthm}
\usepackage{amsfonts}
\usepackage{amssymb}
\usepackage{eucal}
\usepackage{color}

\input prepictex
\input pictex
\input postpictex

\newtheorem{theorem}{Theorem}[section]
\newtheorem{lemma}[theorem]{Lemma}
\newtheorem{proposition}[theorem]{Proposition}

\theoremstyle{definition}
\newtheorem{definition}[theorem]{Definition}

\newtheorem{remark}[theorem]{Remark}

\numberwithin{equation}{section}

\newcommand{\la}{\lambda}
\newcommand{\Ber}{\text{Ber}}
\newcommand{\ad}{\text{ad}}
\newcommand{\C}{ \mathbb C }

\newcommand{\Der}{\text{Der}}
\newcommand{\Dist}{\text{Dist}}

\newcommand{\ev}{\text{\rm ev}}
\newcommand{\GL}{{\text{GL}}}
\newcommand{\gl}{{\mathfrak{gl}}}
\newcommand{\Hom}{\text{Hom}}

\newcommand{\Lie}{\text{Lie}}

\newcommand{\salg}{{\mathfrak{salg}}}
\newcommand{\spo}{{{spo}}}
\newcommand{\SpO}{{{SpO}}}
\newcommand{\Z}{ \mathbb Z }

\newcommand{\cj}{{\mathcal J}}

{\vskip-\lastskip\medskip
  \noindent
  {\em #1.}\enspace
  }%
{\qed\par\medskip
  }

\begin{document}
\title[Modular representations of supergroups]
{Modular representations of the ortho-symplectic supergroups}

\author{Bin Shu}
\address{Department of Mathematics, East China Normal University,
Shanghai 200062,  China} \email{bshu@euler.math.ecnu.edu.cn}

\author[Weiqiang Wang]{Weiqiang Wang}
\address{Department of Mathematics, University of Virginia,
Charlottesville, VA 22904} \email{ww9c@virginia.edu}

\begin{abstract}
A Chevalley type integral basis for the ortho-symplectic Lie
superalgebra is constructed. The simple modules of the
ortho-symplectic supergroup over an algebraically closed field of
prime characteristic not equal to $2$ are classified, where a key
combinatorial ingredient comes from the Mullineux conjecture on
modular representations of the symmetric group. A Steinberg
tensor product theorem for the ortho-symplectic supergroup is also
obtained.
\end{abstract}

\maketitle
\date{}
%\setcounter{section}{-1}

%%
%%
%%
%%
%%
%%
%\tableofcontents

\section{Introduction}
\subsection{}
%The background}

Lie superalgebras, supergroups, and their representation theory
over the field of complex numbers $\C$ have been studied
extensively in literature since the classification of
finite-dimensional complex simple Lie superalgebras by Kac
\cite{Kac}. More on supergroups and supergeometry over $\C$ can be
found in the book of Manin \cite{Ma}. In recent years, the modular
representations of algebraic supergroups $GL(n|m)$ and $Q(n)$ over
an algebraically closed field $k$ of characteristic $p\neq 2$ have
been initiated by Brundan, Kleshchev and Kujawa \cite{BK1, BK2,
BKu, Ku}.

The modular representation theory of {\em super}groups not only is
of intrinsic interest in its own right (with the rich classical
results in representations of algebraic groups \cite{Jan} in
mind), but also has found remarkable applications to classical
mathematics: the classification of simple modules of the spin
symmetric group over $k$ in \cite{BK1} using $Q(n)$, and a new
conceptual proof in \cite{BKu} using $GL(n|m)$ of the celebrated
Mullineux conjecture \cite{Mu} which describes the correspondence
of simple modules of the symmetric group $S_n$ over $k$ upon
tensoring with the $1$-dimensional sign module. The classification
of the simple $Q(n)$-modules  was also nontrivial \cite{BK2}, in
contrast to the algebraic group setup (cf. Jantzen \cite{Jan}).
\subsection{}
%}

The goal of this paper is to initiate the study of modular
representations of the ortho-symplectic supergroup $SpO(2n|\ell)$
over an algebraically closed field $k$ of characteristic $p>2$. We
construct an integral basis (called Chevalley basis as usual) for
Lie superalgebra $spo(2n|\ell)$ and classify the simple modules of
the algebraic supergroup $SpO(2n|\ell)$ for every $n$ and $\ell$.

Recall that the ortho-symplectic Lie superalgebra $spo(2n|\ell)$,
which contains $sp(2n) \oplus so(\ell)$ as its even Lie
subalgebra, provides the other infinite-series classical
superalgebras besides type $A$ in the list of \cite{Kac}. Let us
exclude the classical Lie algebras by assuming $n \ge 1$ and $\ell
\ge 1$ here. The infinite series $spo(2n|\ell)$ is further divided
into four infinite families by root systems: the series $B(0,n)$
corresponding to $\ell=1$; the series $C(n)$ corresponding to
$\ell=2$; the series $B(m,n)$ for $\ell=2m+1$ and the series
$D(m,n)$ for $\ell=2m$, where $m \ge 1$. Already over $\C$, the
finite-dimensional representation theory of $B(m,n)$ and $D(m,n)$
is very challenging and remains to be better understood (see
Serganova \cite{Ser}). The nontrivial classification of
finite-dimensional simple $spo(2n|\ell)$-modules was obtained in
\cite{Kac}.

\subsection{}
%The layout}

In Section~\ref{sec:equiv}, we establish by a simple and uniform
approach an equivalence of categories of rational $G$-modules and
of locally finite $(Dist(G),T)$-modules under some natural
assumptions on an algebraic supergroup $G$, where $T$ is a maximal
torus of $G$ and $Dist(G)$ denotes the superalgebra of
distributions of $G$. The verification of the assumptions for the
equivalence of categories theorem for $SpO(2n|\ell)$ will follow
from results in Section~\ref{sec:chevalley}. These assumptions can
be easily checked for supergroups $GL(n|m)$ and $Q(n)$. Such an
equivalence of categories were earlier established in \cite{BK2}
for $Q(n)$ and later in \cite{BKu} for $GL(n|m)$ by obtaining in
an elementary yet ad hoc case-by-case method an isomorphism
between a restricted dual of the superalgebra of distributions and
the coordinate superalgebra of the supergroup.

We then introduce in Section~\ref{sec:chevalley} a Chevalley basis
for the Lie superalgebra $spo(2n|\ell)$ for any $n, \ell$. While
our constructions of the Chevalley basis are carried out
explicitly case by case, we observe a uniform phenomenon quite
similar to the characterization of Chevalley basis for simple Lie
algebras (cf. Steinberg \cite[Theorem~1]{St2}). The Chevalley
basis further leads to a Kostant (integral) basis for the
superalgebra of distributions $Dist(SpO(2n|\ell))$. In
establishing these integrality statements, we encounter a new
phenomenon where twice of an (odd) root could sometimes be an
(even) root for $spo(2n|\ell)$.

In Section~\ref{sec:tensor}, we establish an analog of the
Steinberg tensor product theorem for the supergroup
$G=SpO(2n|\ell)$. Once we have the equivalence of categories in
place (see Section~\ref{sec:equiv}), our approach is quite
parallel to \cite{BK2, BKu, Ku}, which in turn followed a strategy
in the algebraic group setup (cf. Cline-Parshall-Scott \cite{CPS}
and \cite{Jan}). The equivalence of categories in
Section~\ref{sec:equiv} allows us to study $G$-modules using the
highest weight module theory of $Dist(G)$. However a new major
difficulty arises (when $\ell \geq 3$): not every simple highest
weight $Dist(G)$-module $L(\la)$ for $\la \in X^+(T)$ is
finite-dimensional, where $X^+(T)$ denotes the set of dominant
integral weights for the underlying even subgroup of $G$.

In Section~\ref{sec:classif}, we determine completely the subset
$X^\dag(T)$ of $X^+(T)$ which parameterize the simple
$SpO(2n|\ell)$-modules. Note that the subset $X^\dag(T)$ differs
from $X^+(T)$ already in characteristic zero (cf. \cite{Kac}).
Remarkably, a key combinatorial ingredient in Mullineux conjecture
singles out the subset $X^\dag(T)$, which depends on the
characteristic $p$ of the ground field in general. We refer to the
Introduction of \cite{BKu} for more references and history on the
solution of Mullineux conjecture by Kleshchev and others (also cf.
\cite{FK, Xu}).

The main tool in the proof of our classification is the method of
odd reflections which has been used over $\C$ by Serganova {\em et
al} (cf. \cite{LSS, PS, Ser}) and then also used by Brundan-Kujawa
\cite{BKu} for $GL(n|m)$ in positive characteristic. To a large
extent, our proof was inspired by the Brundan-Kujawa
classification of simple {\em polynomial} $GL(n|m)$-modules, i.e.
the simple subquotients appearing in various tensor powers of the
natural $GL(n|m)$-module. It is a remarkable and puzzling
coincidence that the simple $SpO(2n|2m+1)$-modules and simple
polynomial $GL(n|m)$-modules are classified by an identical set of
weights. Our classification which is valid over $\C$ can be shown
to be equivalent to the classification in \cite{Kac} over $\C$
using Dynkin labels where a totally different argument was
sketched.

%
%
%
%\subsection{}
\vspace{.3cm}

\noindent {\bf Acknowledgments.} B.S. is partially supported from
CSC, NSF and PCSIRT of China. He thanks the Institute of
Mathematics and Department of Mathematics at University of
Virginia for the hospitality and support during his visit in
2004--05. W.W. is partially supported by NSF and NSA grants, and
he thanks Jon Kujawa for a stimulating and helpful discussion.

{\em Convention.} The terminology of ideals, subalgebras, modules
etc. of a Lie superalgebra instead of superideals,
subsuperalgebras, supermodules, etc. is adopted in this paper.

\section{An equivalence of module categories}
\label{sec:equiv}

\subsection{Algebraic supergroups}

We first briefly recall the generalities on algebraic supergroups,
following \cite[Section~2]{BK2}, (which is in turn a
generalization of the approach by Demazure-Gabriel and Jantzen
\cite{Jan}), also cf. \cite{Ma}.

Let $k$ be a fixed algebraically closed field of characteristic $p
\neq 2$. All objects in this paper will be defined over $k$ unless
otherwise specified. Let $A =A_{{\bar{0}}} + A_{{\bar{1}}}$ be a
commutative superalgebra (i.e. $\Z_2$-graded algebra) over $k$,
i.e. $ab=(-1)^{|a||b|}ba$ for all homogeneous elements $a,b\in A$
of degree $|a|,|b|\in \Z_2$. In the sequel, we assume that all
formulas are defined via the homogeneous elements and extended by
linearity. An element in $A_{\bar{0}}$ (resp. $A_{\bar{1}}$) is
called {\em even} (resp. {\em odd}). From the supercommutativity
it follows that $a^2 =0$ for all $a\in A_{{\bar{1}}}.$ We will
denote by $\salg$ the category of commutative superalgebras over
$k$ and even homomorphisms. A fundamental object in $\salg$ is the
free commutative superalgebra $k [x_1, \ldots, x_n; \xi_1, \ldots,
\xi_m]$ in even generators $x_i$ and odd generators $\xi_j$.

An affine superscheme $X$ will be identified with its associated
functor in the category of superschemes
$$\Hom ({Spec} (-), X): \salg \longrightarrow \mathfrak{sets}. $$
The affine superscheme $\mathbb A^{n|m} := {Spec}\; k [x_1,
\ldots, x_n; \xi_1, \ldots, \xi_m]$ as a functor sends a
commutative superalgebra $A$ to $\mathbb A^{n|m}(A)
=A_{{\bar{0}}}^n + A_{{\bar{1}}}^m.$ For an affine superscheme
$X$, its coordinate superalgebra $k[X]$ is the superalgebra
$\text{Mor} (X, \mathbb A^{1|1})$ of all natural transformations
from the functor $X$ to $\mathbb A^{1|1}$. One has
$X=\Hom_\salg(k[X],-)$.

An affine algebraic supergroup $G$ is a functor from the category
$\salg$ to the category of groups, which associates to a
commutative superalgebra $A$ a group $G(A)$ functorially, and
whose coordinator algebra $k[G]$ is finitely generated. For an
algebraic supergroup $G$, $k[G]$ admits a canonical structure of
Hopf superalgebra, with comultiplication $\Delta: k[G] \rightarrow
k[G]\otimes k[G]$, the antipode $S: k[G] \rightarrow k[G]$, and
the counit $\varepsilon: k[G]\rightarrow k$. Set
$$
\cj := \ker(\varepsilon).
$$
A closed subgroup of $G$ is an affine supergroup scheme whose
coordinate algebra is a quotient of $k[G]$ by a Hopf ideal $I$. In
particular, the underlying purely even group of $G$, denoted by
$G_{\text {ev}}$, corresponds to the Hopf ideal
$k[G]k[G]_{\bar{1}}$. That is, $k[G_\ev] \cong k[G]\slash
k[G]k[G]_{\bar{1}}.$

In the remainder of the paper, an affine algebraic supergroup will
simply be referred to as a supergroup.

\subsection{Superalgebra of distributions}

Let $G$ be a supergroup. The superspace of distributions (at the
identity $e\in G$) is
$$\Dist(G) :=\cup_{n\geq 0}\Dist_n(G)$$
where $\Dist_n(G) :=\{X\in k[G]^* \mid X(\cj^{n+1})=0\} \cong
(k[G]\slash \cj^{n+1})^*$.
%
%Similarly, we have $\Dist^+(G):=\cup_{r\geq 0}\Dist^+_r(G)$ for
%$\Dist^+_r(G)=\{X\in \Dist_r(G)\mid X(1)=0\}\cong (\cj\slash
%\cj^{r+1})^*$.

The space $\Dist(G)$ is  a cocommutative Hopf superalgebra whose
multiplication $*$ is dual to the comultiplication $\Delta$ of
$k[G]$ just as in the ordinary case (cf. \cite{BK2} and
\cite{Jan}). Furthermore, $\Dist(G)$ is a filtered superalgebra
given by:
$$k\subset\Dist_1(G)\subset\cdots \subset
\Dist_r(G)\subset\Dist_{r+1}(G)\subset\cdots.
$$
For $f_1,\cdots,f_n\in \cj$ and $n\in \mathbb N$, we have
\begin{equation*}
\Delta(f_1\cdots f_n)\in \prod_{i=1}^n(1\otimes f_i+f_i\otimes
1)+\sum_{r=1}^n\cj^r\otimes \cj^{n+1-r}.
\end{equation*}
It follows that for $X\in \Dist_s(G)$ and $Y\in \Dist_t(G)$
\begin{equation*}\label{DISTT}
[X,Y] :=X * Y-(-1)^{|X||Y|}Y* X\in \Dist_{s+t-1}(G).
\end{equation*}
Hence, the tangent space at the identity
\begin{equation*}
T_e(G) := \{ X\in \Dist_1(G) \mid X(1)=0\} \cong (\cj\slash
\cj^2)^*
\end{equation*}
carries a Lie superalgebra structure; it is called the Lie
superalgebra of $G$ and will be denoted by $\Lie(G)$.

\subsection{The restricted structure on $Lie(G)$} \label{JEVEN}
\begin{definition}
A Lie superalgebra ${\mathfrak g} ={\mathfrak
g}_{\bar{0}}+{\mathfrak g}_{\bar{1}}$ is called a {\em restricted}
Lie superalgebra (or $p$-Lie superalgebra), if the following
conditions are satisfied:
\begin{enumerate}
\item[(a)] ${\mathfrak g}_{\bar{0}}$ is a restricted Lie algebra with
$p$-mapping $[p]: {\mathfrak g}_{\bar{0}} \rightarrow {\mathfrak
g}_{\bar{0}}$ \cite[Chap.~4]{Jac}.
\item[(b)] ${\mathfrak g}_{\bar{1}}$ is a restricted ${\mathfrak
g}_{\bar{0}}$-module via the adjoint action, i.e.
$\ad(X^{[p]})(X_1) =\ad(X)^p(X_1)$, for $X \in {\mathfrak
g}_{\bar{0}}, X_1\in {\mathfrak g}_{\bar{1}}$.
\end{enumerate}
\end{definition}

Let $G$ be a supergroup. The canonical map $\pi: k[G] \rightarrow
k[G_\ev] = k[G]/k[G]k[G]_{\bar{1}}$ sends $\cj$ to the kernel
$\cj_\ev$ of $\varepsilon_\ev:  k[G_\ev] \rightarrow k$ and
$\pi(\cj^i)\subset \cj_\ev^i$ for $i\ge 1$. This induces an
injective algebra homomorphism $ \pi^*: \Dist(G_\ev) \rightarrow
\Dist(G)$. $\pi$ also induces an isomorphism of vector spaces from
$(\cj\slash \cj^2)_{\bar{0}}$ to $\cj_\ev\slash \cj_\ev^2$, and
both spaces are isomorphic to the space $\cj_{\bar{0}}\slash
(\cj_{\bar{0}}^2 +k[G]_{\bar{1}}^2)$. Thus, we have the following.
\begin{lemma}\label{ISOEVENGP}
The superalgebra homomorphism $\pi^*$ induces an isomorphism of
Lie algebras from $\Lie(G_\ev)$ onto
$\Lie(G)_{\bar{0}}=\Lie(G)\cap \Dist(G)_{\bar{0}}$.
\end{lemma}

\begin{proposition} \label{RESPROP}
Let $G$ be a supergroup. Then, $\Lie(G)$ is a restricted Lie
superalgebra with the $p$-mapping: $X\mapsto X^{[p]}$ for $X\in
\Lie(G)_{\bar{0}}$, where $X^{[p]} :=\overbrace{X*\cdots *X}^p$ is
defined in $\Dist(G)$. Moreover, the restricted structure on
$\Lie(G_\ev)$ as a subalgebra of $\Lie(G)$ coincides with the one
induced as Lie algebra of the algebraic group $G_\ev$.
\end{proposition}
\begin{proof}
Since $G_\ev$ is an algebraic group, $\Lie(G_\ev)$ is a restricted
Lie algebra with the $p$-mapping given by the $p$-th power in
$\Dist(G_\ev)$ (cf. \cite{Bor, Jan}). The compatibility of the
restricted structures on $\Lie(G)$ and on $\Lie(G_\ev)$ now
follows from Lemma~\ref{ISOEVENGP} since $ \pi^*: \Dist(G_\ev)
\rightarrow \Dist(G)$ is an algebra homomorphism.

For $X_0\in \Lie(G)_{\bar{0}}, X_1\in \Lie(G)_{\bar{1}}$, and $r
\ge 1$, we have
\begin{equation}\label{JACOB}
(\ad X_0)^r(X_1)=\sum_{i=0}^r(-1)^{r-i}{r\choose i}
 X_0^i*X_1*X_0^{r-i}.
\end{equation}
This implies that $\ad(X_0)^p(X_1)=X_0^{[p]}*X_1-X_1 * X_0^{[p]}
=\ad(X_0^{[p]})(X_1)$. Thus the $\Lie(G)_{\bar{0}}$-module
$\Lie(G)_{\bar{1}}$ via adjoint action is a restricted module.
\end{proof}

For $i=0,1$, we let
$$\Der_k(k[G], k)_{\bar{i}} :=\{X\in \Hom_k(k[G],k)_{\bar{i}} \mid
X(fg) =X(f)\varepsilon(g) + (-1)^{i|f|} \varepsilon(f)X(g)\}$$
and let $\Der_k(k[G], k) =\Der_k(k[G], k)_{\bar{0}} \oplus
\Der_k(k[G], k)_{\bar{1}}$. The following can be established as in
the case of algebraic groups (cf. \cite{Bor}).
\begin{lemma}\label{POINTD}
As restricted Lie superalgebras, $Lie(G)\cong \Der_k(k[G],k)$.
\end{lemma}

\subsection{$G$-modules and $Dist(G)$-modules}

For any vector superspace $M$, we have the general linear
supergroup $GL(M)$ which sends each commutative superalgebra $A$
to $GL(M,A)$, the group of all even $A$-linear automorphisms of
$M\otimes A$. A representation of a supergroup $G$ or a (rational)
$G$-module, means a natural transformation $\rho:G\rightarrow
GL(M)$ for some vector superspace $M$. As for algebraic groups, a
representation $M$ of $G$ is equivalent to a right $k[G]$-comodule
structure on $M$ given by an even map $\eta_M: M \rightarrow M
\otimes k[G]$ (cf. \cite{Jan, BK2}). We denote by
$G$-$\mathfrak{mod}$ the category of rational $G$-modules with
(not necessarily homogeneous) $G$-homomorphisms. Note that a
$G$-module is always {\em locally finite}, i.e., it is a sum of
finite-dimensional $G$-modules.

Given a closed subgroup $H$ of $G$, a $Dist(G)$-module $M$ is
called a $(Dist(G),H)$-module if $M$ has also a structure of
$H$-module such that the $Dist(H)$-module structures on $M$
induced from the actions of $Dist(G)$ and of $H$ coincide. We
denote by $(Dist(G),H)$-$\mathfrak{mod}$ the category of locally
finite $(Dist(G),H)$-modules.

There is a natural functor
$$
\Psi: G\mathfrak{-mod} \longrightarrow (Dist(G),H)\mathfrak{-mod}
$$
as follows: one endows a $G$-module
$M$ with an action of $\mu \in Dist(G)$ by $(id_M \otimes \mu)
\circ \eta_M$, and $M$ is acted by $H$ as a subgroup of $G$.

\subsection{An equivalence of categories}

Let $G$ be a supergroup. The Frobenius morphism $F:G\rightarrow G$
is the natural transformation which assigns to each $A \in \salg$
the morphism $F(A): G(A)\rightarrow G(A)$ induced by $a \mapsto
a^{p^r}$ for $a \in A$.
%for which $F(A)$ raises each matrix entry to the $p$-th power.
The image of $F$ lies in $G_\ev$ and we often denote $F:
G\rightarrow G_\ev$. For $r\geq 1$, we define $F^r:G\rightarrow
G_\ev$ by the $r$-th iteration of $F$, whose kernel $G_r$ is
called the $r$-th Frobenius kernel of $G$. Since $G_1$ is a normal
subgroup of $G$, $G_\ev G_1 :=\{ g_0g_1 |g_0 \in G_\ev, g_1 \in
G_1\}$ is a subgroup of $G$.

\begin{lemma} \label{lem:inters}
For every $r \ge 1$, $G_r \cap G_\ev =(G_\ev)_r$.
\end{lemma}
\begin{proof}
Note that $G_r \cap G_\ev$ is the kernel of $F^r|_{G_\ev}$, the
restriction to ${G_\ev}$ of the Frobenius morphism $F^r$ on $G$.
Now the lemma follows from that $F^r|_{G_\ev}$ coincides with the
$r$-th Frobenius morphism of $G_\ev$.
\end{proof}

A key argument for the following lemma was supplied by Jon Kujawa.
\begin{lemma} \label{lem:keygen}
Assume $G$ is a supergroup with its even subgroup $G_\ev$ defined
over $\mathbb F_p$. Then, $G =G_\ev G_1$. Furthermore, $G_r
=(G_\ev)_r G_1$ for every $r \ge 1$.
\end{lemma}

\begin{proof}
The restriction $F|_{G_\ev}: G_\ev \rightarrow G_\ev$ coincides
with the Frobenius morphism of $G_\ev$ which is surjective as
$G_\ev$ is defined over $\mathbb F_p$ (cf. \cite{Jan}). Recall
that for any $g \in G$, we have $F(g) \in G_\ev$. By the
surjectivity of the Frobenius morphism $F|_{G_\ev}$ on $G_\ev$,
there exists an element $g_0 \in G_\ev$ so that $F(g_0) =F(g)$.
Thus, $F( g_0^{-1} g) =1$, i.e. $g_1: = g_0^{-1} g\in G_1$, and $g
=g_0g_1 \in G_\ev
G_1$. %This proves $G =G_\ev G_1$.

Now if $g \in G_r$, then $ 1= F^r(g) = F^r(g_0g_1) = F^r(g_0)$,
that is $g_0\in G_r$. It follows by Lemma~\ref{lem:inters} that
$g_0 \in (G_\ev)_r$.
\end{proof}

Since $G_1$ is a finite (super)group scheme and $k[G_1]$ is a
finite-dimensional Hopf superalgebra, we have the next lemma
following \cite[Chap. I.8]{Jan}. Recall for a $G_1$-module $M$, we
have the comodule structure map $\eta_M: M \rightarrow M \otimes
k[G_1]$. Denote by $G_1$-$\mathfrak{mod}$ the category of
$G_1$-modules and by $Dist(G_1)$-$\mathfrak{mod}$ the category of
$Dist(G_1)$-modules. Note that $Dist(G_1) = k[G_1]^*$.

\begin{lemma} \label{lem:equivG1}
The functor $\Psi_1:$ $G_1$-$\mathfrak{mod} \rightarrow
Dist(G_1)$-$\mathfrak{mod}$, which endows a $G_1$-module $M$ with
an action of $\mu \in Dist(G_1)$ by $(id_M \otimes \mu) \circ
\eta_M$, is an equivalence of categories.
\end{lemma}

For the remainder of this subsection, we make the following
assumption on $G$.

\vspace{.2cm} \noindent {\bf Assumptions.} {\em $G$ is an
algebraic supergroup whose even subgroup $G_\ev$ is a connected
reductive group defined over $\mathbb F_p$ with a maximal torus
$T$. Furthermore, there is an integral basis for $Lie(G)$ and a
corresponding basis for $Dist(G)$ which extend the Chevalley basis
for $Lie(G_\ev)$ and the Kostant basis for $Dist(G_\ev)$
respectively.}

It is known (cf. \cite{Jan}) that under the above assumption on
the algebraic group $G_\ev$ the natural functor $\Psi_\ev:$
$G_\ev$-$\mathfrak{mod}$ $\rightarrow
(Dist(G_\ev),T)$-$\mathfrak{mod}$ (defined just as the functor
$\Psi$) is an equivalence of categories. Clearly $\Psi:$
$G$-$\mathfrak{mod}$ $\rightarrow (Dist(G),T)$-$\mathfrak{mod}$ is
compatible with $\Psi_\ev$ and also with $\Psi_1$ via forgetful
functors.

\begin{theorem} \label{th:equiv}
Retain the above assumptions on $G$. Then $\Psi$ is an equivalence
of categories between $G$-$\mathfrak{mod}$ and
$(Dist(G),T)$-$\mathfrak{mod}$.
\end{theorem}

\begin{proof}
We define a functor $\widetilde{\Psi}$ from
$(Dist(G),T)$-$\mathfrak{mod}$ to $G$-$\mathfrak{mod}$ by lifting
every locally finite $(Dist(G),T)$-module to a $G$-module as
follows. Let $M$ be a locally finite $(Dist(G),T)$-module. Noting
$Dist(G_\ev) \subset Dist(G)$ as subalgebras and regarding $M$ as
a $(Dist(G_\ev),T)$-module, we can lift $M$ to an $G_\ev$-module
canonically as in \cite[pp.171]{Jan}. In the same vein, noting
$Dist(G_1) \subset Dist(G)$ as subalgebras and regarding $M$ as a
$Dist(G_1)$-module, we can lift $M$ to an $G_1$-module  via an
inverse functor of $\Psi_1$ (cf. Lemma~\ref{lem:equivG1}). The $M$
endowed with these two lifted structures coincide as
$(G_\ev)_1$-modules since both are lifted canonically from the
same $Dist((G_\ev)_1)$-module structure, and thus they coincide as
$G_1 \cap G_\ev$-modules by Lemma~\ref{lem:inters}. By
Lemmas~\ref{lem:inters} and \ref{lem:keygen}, we obtain a
well-defined $G$-module structure on $M$ by letting any $g =g_0g_1
\in G$ with $g_0\in G_\ev$ and $g_1\in G_1$ acts by composing the
actions of $g_0$ and $g_1$.

Clearly, $\Psi$ and $\widetilde{\Psi}$ are inverses of each other.
\end{proof}

\begin{remark}
It is straightforward to check the assumptions in
Theorem~\ref{th:equiv} for supergroups $GL(m|n)$ and $Q(n)$, cf.
\cite{BK2, BKu}. Thus, our Theorem~\ref{th:equiv} gives a simpler
and uniform proof of the equivalence of categories for these two
supergroups obtained earlier in {\em loc. cit.} by first
establishing an isomorphism between a restricted dual of the
superalgebra of distributions and the coordinate superalgebra. We
shall also see in the next section that the supergroups of type
$SpO$ satisfy these assumptions as well, and it is not clear if
the method in {\em loc. cit.} is applicable for these new
supergroups.
\end{remark}

\section{The Chevalley basis of Lie superalgebra $spo(2n|\ell)$}
\label{sec:chevalley}

\subsection{The supergroup $\SpO(2n|\ell)$}

The supergroup $\GL(r|s)$ is the functor which associates to any
$A \in \salg$ the group $GL(r|s;A)$ of all invertible $(r+s)
\times (r+s)$ matrices of the form
\begin{eqnarray} \label{GLmatrix}
g =\left[
    \begin{array}{rr}
        a  &  b   \\
        c  &  d
    \end{array}
    \right]
\end{eqnarray}
where $a$ (resp. $d$) is an $r \times r$ (resp. $s\times s$)
matrix with entries in $A_{\bar{0}}$, $b$, $c$ is $r \times s$
(resp. $s \times r$) with entries in $A_{\bar{1}}$. The
supertranspose of $g$ is defined as
\begin{eqnarray*}
g^{st} :=\left[
    \begin{array}{rr}
        a^t  &  c^t   \\
       -b^t  &  d^t
    \end{array}
    \right]
\end{eqnarray*}
where the superscript $t$ denotes the transpose of a matrix in the
usual sense. It is well known (cf. e.g. \cite{Ma}) that $g$ is
invertible if and only if both $a$ and $d$ are invertible. Let
$Mat_{r|s}$ be the affine superscheme with $Mat_{r|s} (A)$
consisting of al $(r+s) \times (r+s)$ matrices of the form
(\ref{GLmatrix}). Then $k[GL(r|s)]$ is the localization of
$k[Mat_{r|s}]$ at the function $\det: g \rightarrow \det a \det
b.$ The Lie superalgebra of $\GL(r|s)$, denoted by $\mathfrak {gl}
(r|s)$, consists of matrices of the form (\ref{GLmatrix}) with
$a,b,c,d \in k$, and the $\Z_2$-grading is defined such that
$\left[
    \begin{array}{rr}
        a  &  0   \\
        0  &  d
    \end{array}
    \right]$ is even and $\left[
    \begin{array}{rr}
        0  &  b   \\
        c  &  0
    \end{array}
    \right]$
is odd.

Recall (cf. \cite[Chap. 3]{Ma}) there is a morphism of supergroups
called the {\em Berezian or superdeterminant}, $\Ber: GL(r|s)
\rightarrow GL(1|0)$ defined as follows: for any $A \in \salg$,
$\Ber: GL(r|s; A) \rightarrow GL(1|0;A)$ sends an element $g$ in
(\ref{GLmatrix}) to
$$\Ber (g) =\det (a-bd^{-1}c) \cdot \det d^{-1}.
$$
It has various favorable properties, e.g.,
$\Ber (g^{st}) =\Ber(g)$.

Define the $(2n+2m+1) \times (2n+2m+1)$ matrix in the
$(n|n|m|m|1)$-block form
\begin{eqnarray} \label{symform}
\mathfrak J_{2n|2m+1} := \left[
    \begin{array}{rrrrr}
        0  &  I_n  & 0   & 0  & 0  \\
     -I_n  &  0    & 0   & 0  & 0  \\
        0  &  0    & 0   &I_m & 0  \\
        0  &  0    & I_m & 0  & 0  \\
        0  &  0    & 0   & 0  & 1
    \end{array}
    \right]
\end{eqnarray}
where $I_n$ is the $n \times n$ identity matrix. Let $\mathfrak
J_{2n|2m}$ denote the $(2n+2m) \times (2n+2m)$ matrix obtained
from $\mathfrak J_{2n|2m+1}$ with the last row and column deleted.
Denote by $\SpO(2n|\ell)$ (with $\ell =2m$ or $2m+1$) the
supergroup functor which associates to any $A \in \salg$ the group
which consists of all $(2n+\ell) \times (2n+\ell)$ matrices of the
form
\begin{eqnarray} \label{eq:spo}
\{g \in GL(2n|\ell;A) \mid g^{st} \mathfrak J_{2n|\ell} \; g =
\mathfrak J_{2n|\ell}, \Ber(g) =1\}.
\end{eqnarray}
Note that the defining relations in (\ref{eq:spo}) are actually
defined over $\mathbb F_p$. The underlying even subgroup is
$$
SpO(2n|\ell)_\ev =SpO(2n|\ell) \cap GL(2n|\ell)_\ev \cong Sp(2l)
\times SO(\ell).
$$

\subsection{Lie superalgebra $\spo(2n|\ell)$}

As in the case of Lie algebras, with the help of
Lemma~\ref{POINTD} we can identify the Lie algebra
$\Lie(\SpO(2n|\ell))$ with
$$ \spo(2n|\ell) := \{g \in \gl (2n|\ell) \mid g^{st} \mathfrak
J_{2n|\ell} + \mathfrak J_{2n|\ell} \, g =0 \}.
$$
The $\spo(2n|2m+1)$ consists of the $(2n+2m+1) \times (2n+2m+1)$
matrices in the following $(n|n|m|m|1)$-block form
\begin{eqnarray} \label{matrixSPO}
g = \left[
    \begin{array}{rrrrr}
    d      &   e     & y_1^t  & x_1^t &  z_1^t       \\
    f      &  -d^t   & -y^t   & - x^t &  -z^t  \\
    x      & x_1     &  a     &  b    &  -v^t  \\
    y      & y_1     &  c     & -a^t  &  -u^t   \\
    z      & z_1    & u      &  v    &    0
    \end{array}
    \right]
\end{eqnarray}
where  $b, c$  are skew-symmetric, and $e, f$ are symmetric
matrices. The Lie superalgebra $\spo(2n|2m+1)$ is called type
$B(m,n)$ in \cite{Kac}. We assume here and below to index the rows
and columns of (\ref{symform}) and (\ref{matrixSPO}) by the finite
set $I(2n|2m+1)$, where
$$I(r|s) :=\{-r, -r+1, \cdots,  -1; 1, 2, \cdots, s\}.$$
Similarly, the Lie superalgebra $\spo(2n|2m)$ consists of the
$(2n+2m) \times (2n+2m)$ matrices which are obtained from $g$ of
the form (\ref{matrixSPO}) with the last row/column deleted and
whose rows/columns are indexed by $I(2n|2m)$.

\subsection{Chevalley basis for $\spo(2n|2m+1)$}
\label{CHEVALB}
 The even subalgebra of $\spo(2n|2m+1)$ is
$sp(2n)\oplus so(2m+1)$. The Cartan subalgebra $\mathfrak h$ of
$\spo(2n|2m+1)$ is taken to be the subalgebra of diagonal
matrices, and it has a basis given by $E_{i-n,i-n} -E_{i,i},
E_{j,j} -E_{j+m,j+m} (-n \le i <0 < j \le m)$, and let $\delta_a,$
where $a \in I(n|m),$ be the dual basis. The standard set $\Pi$ of
simple roots is:
$$\{\delta_{-n} -\delta_{-n+1}, \ldots, \delta_{-2} -\delta_{-1},
\delta_{-1} - \delta_1, \delta_1 -\delta_2, \ldots, \delta_{m-1}
-\delta_m, \delta_m\}  \text{ for } m\ge 1,$$
with $\delta_{-1} - \delta_1$ being odd. The corresponding Dynkin
diagram is (where the node $\otimes$ denotes an odd simple root
twice of which is not root):
\begin{equation*}
%Dynkin diagram for $spo(2n|2m+1)$
\begin{array}{cc}
\setlength{\unitlength}{0.3in}
\begin{picture}(10,2)
\put(-3,1){\makebox(0,0)[c]{$\circ$}}
\put(1.5,1){\makebox(0,0)[c]{$\circ$}}
\put(3.5,1){\makebox(0,0)[c]{$\otimes$}}
\put(5.5,1){\makebox(0,0)[c]{$\circ$}}
\put(10,1){\makebox(0,0)[c]{$\circ$}}
\put(11.4,1){\makebox(0,0)[c]{$\circ$}}
 \put(-2.85,1){\line(1,0){1.15}}
 \put(-1.45,.96){{$\ldots\ldots$}}
  \put(0.2,1){\line(1,0){1.15}}
 \put(1.65,1){\line(1,0){1.62}}
 \put(3.70,1){\line(1,0){1.62}}
 \put(5.65,1){\line(1,0){1.15}}
 \put(7,.96){{$\ldots\ldots$}}
  \put(8.65,1){\line(1,0){1.17}}
 \put(10.1,0.85){$=$}
 \put(10.45,0.85){$\Longrightarrow$}
\put(-3,0.4){\makebox(0,0)[c]{$\delta_{-n}-\delta_{1-n}$}}
\put(1.1,0.4){\makebox(0,0)[c]{$\delta_{-2} -\delta_{-1}$}}
\put(3.5,0.4){\makebox(0,0)[c]{$\delta_{-1} -\delta_{1}$}}
\put(5.7,0.4){\makebox(0,0)[c]{$\delta_{1} -\delta_{2}$}}
\put(10,0.4){\makebox(0,0)[c]{$\delta_{m-1} -\delta_{m}$}}
\put(11.9,0.4){\makebox(0,0)[c]{$\delta_m$}}
\end{picture}
\end{array}
\vspace{.7cm}
\end{equation*}
If $m=0$, then the standard set $\Pi$ of simple roots is
$$\Pi =\{\delta_{-n} -\delta_{-n+1}, \ldots, \delta_{-2} -\delta_{-1},
\delta_{-1}\}, $$
and its Dynkin diagram is (where the node $\bullet$ denotes an odd
simple root twice of which is a root):
\begin{equation*}
%Dynkin diagram for $spo(2n|1)$
\begin{array}{cc}
\setlength{\unitlength}{0.3in}
\begin{picture}(10,2)
\put(1,1){\makebox(0,0)[c]{$\circ$}}
\put(5.5,1){\makebox(0,0)[c]{$\circ$}}
\put(6.9,1){\makebox(0,0)[c]{$\bullet$}}
 \put(1.15,1){\line(1,0){1.15}}
 \put(2.55,.96){{$\ldots\ldots$}}
 \put(4.2,1){\line(1,0){1.15}}
 \put(5.6,.85){$=$}
 \put(5.9,.85){$\Longrightarrow$}
\put(1,0.4){\makebox(0,0)[c]{$\delta_{-n}-\delta_{1-n}$}}
\put(5.4,0.4){\makebox(0,0)[c]{$\delta_{-2} -\delta_{-1}$}}
\put(7.2,0.4){\makebox(0,0)[c]{$\delta_{-1}$}}
\end{picture}
\end{array}
\end{equation*}
\vspace{.5cm}

The set of roots are $\Delta = \Delta_0 \cup \Delta_1$, a union of
sets of even roots $\Delta_0$ and odd roots $\Delta_1$, and
$\Delta = \Delta^+ \cup - \Delta^+$. Set $\Delta_i = \Delta^+_i
\cup - \Delta^+_i$ for $i=0,1$. More explicitly,
\begin{eqnarray*}
\Delta^+_0
 %&=& \Delta^+ (sp(2n)) \cup \Delta^+ (o(2m+1)) \\
 &=& \{\delta_i \pm \delta_j, -n \le i<j<0 \text{ or } 0 <i <j \le m\} \\
 &&\quad \cup \{2 \delta_i, -n \le i <0 \} \cup \{ \delta_j, 0 < j \le m\},  \\
\Delta^+_1 &=& \{ \delta_i \pm \delta_j, -n \le i<0 <j \le m \}
\cup \{\delta_i, -n \le i <0\}.
\end{eqnarray*}

Note that $\Delta^+$ is compatible with a set of positive roots of
$sp(2n)\oplus so(2m+1)$ whose fundamental system is given by:
$$\Pi_\ev := \{\delta_i -\delta_{i+1}, -n \le i < -1 ;
2\delta_{-1}\} \cup \{\delta_i -\delta_{i+1},  1 \le i <m;
\delta_m\}.
$$
It is understood that for $m=0$ the undefined $\delta_0$ is
omitted from $\Pi_\ev$.

Recall (cf. \cite{St2}) that $sp(2n)\oplus so(2m+1)$ admits a
Chevalley basis (unique up to signs) $\{H_{s} (s \in \Pi_\ev),
X_\alpha (\alpha \in \Delta_0)\}$, whose structure constants are
integers. One of the requirements \cite[Theorem~1]{St2} is that
\begin{equation} \label{eq:axiom}
[X_\alpha, X_\beta] = \pm (r+1) X_{\alpha +\beta}, \text{ if }
\alpha, \beta, \alpha +\beta \in \Delta_0,
\end{equation}
where $r$ is determined by the $\alpha$-string of roots through
$\beta$: $\{\beta + i\alpha \mid -r \le i \le q \}$.

\begin{remark} \label{shortvector}
Associated to the even short root $\delta_j$ with $0 < j \le m$,
we fix the sign and take the Chevalley root vector
%
%\begin{eqnarray*}
$X_{\delta_j} =\sqrt{2} (E_{j, 2m+1} -E_{2m+1, j+m})$
%\end{eqnarray*}
where $\sqrt{2} \in k$. The other Chevalley (long) root vectors
for $sp(2n)\oplus so(2m+1)$ always have entries $0, \pm 1$ in the
standard matrix form (\ref{matrixSPO}).
\end{remark}

Now we define the odd root vectors (indexed by their corresponding
roots): for $-n \le i<0<j \le m$,
%\begin{eqnarray*}
%X_{\delta_i +\delta_j} &:=& E_{j,i} +E_{i-n,j+m} \\
%%
%X_{-\delta_i -\delta_j} &:=&  E_{j+m, i-n} -E_{i,j} \\
%%
%X_{\delta_i -\delta_j} &:=&  E_{j+m,i} + E_{i-n,j} \\
%%
%X_{-\delta_i +\delta_j} &:=&  E_{j,i-n} -E_{i,j+m}\\
%%
%X_{\delta_i} &:=&  \sqrt{2} (E_{2m+1,i} +E_{i-n,2m+1}) \\
%%
%X_{-\delta_i} &:=& \sqrt{2} (E_{2m+1, i-n} -E_{i,2m+1}).
%\end{eqnarray*}
%
\begin{eqnarray}
X_{\delta_i +\delta_j} := E_{j,i} +E_{i-n,j+m}, &&
X_{-\delta_i -\delta_j} :=  E_{j+m, i-n} -E_{i,j} ,\label{oddvec1}\\
X_{\delta_i -\delta_j} :=  E_{j+m,i} + E_{i-n,j}, &&
X_{-\delta_i +\delta_j} :=  E_{j,i-n} -E_{i,j+m},  \label{oddvec2}\\
X_{\delta_i} :=  \sqrt{2} (E_{2m+1,i} +E_{i-n,2m+1}), &&
X_{-\delta_i} := \sqrt{2} (E_{2m+1, i-n} -E_{i,2m+1}).
\label{oddvec3}
\end{eqnarray}

\begin{proposition}
The Chevalley basis elements $\{H_{s} \,(s \in \Pi_\ev), X_\alpha
\, (\alpha \in \Delta_0)\}$ for the even subalgebra $sp(2n)\oplus
so(2m+1)$ together with $X_\alpha \,(\alpha \in \Delta_1)$ in
(\ref{oddvec1}-\ref{oddvec3}) form a basis for Lie superalgebra
$\spo(2n|2m+1)$ with integer structure constants.
\end{proposition}
This basis will be called the {\em Chevalley basis} for
$\spo(2n|2m+1)$.
\begin{proof}
Follows by a direct computation.
\end{proof}

\subsection{Chevalley basis for $\spo(2n|2m), m\ge 2$}
\label{CHEBASIS}
 We continue to regard Lie superalgebra
$\spo(2n|2m)$ as a subalgebra of Lie superalgebra $\spo(2n|2m+1)$
in the matrix form (\ref{matrixSPO}). The even subalgebra of
$\spo(2n|2m)$ is $sp(2n) \oplus so(2m)$. The Cartan subalgebra of
$\spo(2n|2m)$ is the same as the Cartan subalgebra of
$\spo(2n|2m+1)$. The standard set $\Pi_\ev$ of simple roots for
$sp(2n)\oplus so(2m)$ is
$$\Pi_\ev := \{\delta_i -\delta_{i+1}, -n \le i < -1 ;
2\delta_{-1}\} \cup \{\delta_i -\delta_{i+1},  1 \le i <m;
\delta_{m-1}+\delta_m\}.$$

Lie superalgebra $\spo(2n|2m)$ with $m \ge 2$ (and $n \ge 1$) is
called type $D(m,n)$ in \cite{Kac}. Its standard set $\Pi$ of
simple roots is:
$$\{\delta_{-n} -\delta_{1-n}, \cdots, \delta_{-2} -\delta_{-1},
\delta_{-1} - \delta_1, \delta_1 -\delta_2, \cdots, \delta_{m-1}
-\delta_m, \delta_{m-1} + \delta_m\},$$
with $\delta_{-1} - \delta_1$ being odd. Its Dynkin diagram is

\vspace{.9cm}
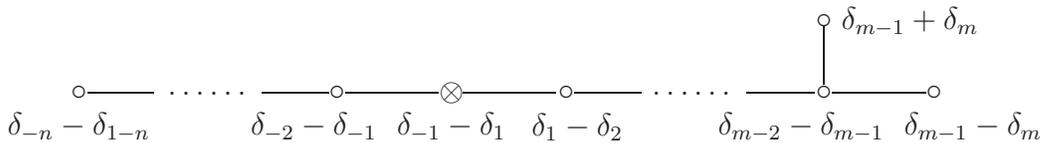
\begin{equation*}
%Dynkin diagram for $spo(2n|2m)$
\begin{array}{cc}
\setlength{\unitlength}{0.3in}
\begin{picture}(10,2)
\put(-3,1){\makebox(0,0)[c]{$\circ$}}
\put(1.5,1){\makebox(0,0)[c]{$\circ$}}
\put(3.5,1){\makebox(0,0)[c]{$\otimes$}}
\put(5.5,1){\makebox(0,0)[c]{$\circ$}}
\put(10,1){\makebox(0,0)[c]{$\circ$}}
\put(11.92,1){\makebox(0,0)[c]{$\circ$}}
\put(10,2.25){\makebox(0,0)[c]{$\circ$}}
 \put(-2.85,1){\line(1,0){1.15}}
 \put(-1.45,.96){{$\ldots\ldots$}}
  \put(0.2,1){\line(1,0){1.15}}
 \put(1.65,1){\line(1,0){1.62}}
 \put(3.70,1){\line(1,0){1.62}}
 \put(5.65,1){\line(1,0){1.15}}
 \put(7,.96){{$\ldots\ldots$}}
 \put(8.65,1){\line(1,0){1.17}}
 \put(10.15,1){\line(1,0){1.62}}
 \put(10,1.15){\line(0,1){1}}
\put(-3,0.4){\makebox(0,0)[c]{$\delta_{-n}-\delta_{1-n}$}}
\put(1.1,0.4){\makebox(0,0)[c]{$\delta_{-2} -\delta_{-1}$}}
\put(3.5,0.4){\makebox(0,0)[c]{$\delta_{-1} -\delta_{1}$}}
\put(5.7,0.4){\makebox(0,0)[c]{$\delta_{1} -\delta_{2}$}}
\put(9.6,0.4){\makebox(0,0)[c]{$\delta_{m-2} -\delta_{m-1}$}}
\put(12.6,0.4){\makebox(0,0)[c]{$\delta_{m-1}-\delta_m$}}
\put(11.5,2.25){\makebox(0,0)[c]{$\delta_{m-1}+\delta_m$}}
\end{picture}
\end{array}
\end{equation*}
\vspace{.4cm}

The corresponding set of positive roots $\Delta^+ =\Delta_0^+ \cup
\Delta_1^+$, where
\begin{align*}
\Delta_0^+
%&= \Delta^+ (o(2m)) \cup \Delta^+ (sp(2n)) %
 &=
 \{\delta_i \pm \delta_j, -n \le i<j<0 \text{ or } 0 <i <j \le m\}
  \cup \{2 \delta_i, -n \le i <0 \}, \\
\Delta_1^+ &= \{\delta_i \pm \delta_j, -n \le i<0< j \le m\}.
\end{align*}
The set of roots is $\Delta = \Delta^+ \cup - \Delta^+$. Set
$\Delta_i = \Delta^+_i \cup - \Delta^+_i$ for $i=0,1$.
%Let $\{H_{s} \,(s \in \Pi_\ev), X_\alpha \, (\alpha \in
%\Delta_0)\}$ be the Chevalley basis for the even subalgebra
%$sp(2n)\oplus o(2m)$ of Lie superalgebra $\spo(2n|2m)$.

\begin{proposition}
Let $m \ge 2$. The Chevalley basis elements for the even
subalgebra $sp(2n)\oplus so(2m)$ together with the odd root
vectors $X_\alpha \,(\alpha \in \Delta_1)$ in
(\ref{oddvec1}-\ref{oddvec2}) form a basis for Lie superalgebra
$\spo(2n|2m)$ with integer structure constants.
\end{proposition}
The proof of this proposition is again by a direct computation.
This basis will be called the {\em Chevalley basis} for
$\spo(2n|2m)$.
\subsection{Chevalley basis for $\spo(2n|2)$}

Lie superalgebra $\spo(2n|2)$ is called type $C(n)$ in \cite{Kac},
and it is more convenient to be realized as $osp(2|2n)$ which
consists of matrices of the $(1|1|n|n)$-block form
\[
\left[
    \begin{array}{rrrr}
    a   &   0  &  x  & x_1  \\
    0   &  -a  &  y  & y_1  \\
  y_1^t & x_1^t&  d  & e  \\
  -y^t  & -x^t  & f  & -d^t
    \end{array}
    \right]
\]
where $e, f$ are symmetric matrices. We index the rows and columns
by $I(2|2n).$  The even subalgebra is $so(2) \oplus sp(2n)$. The
Cartan subalgebra $\mathfrak h$ has a basis given by $E_{-2,-2}
-E_{-1,-1}, E_{i,i} -E_{n+i,n+i} (1\le i \le n)$, and let
$\delta_i \in \mathfrak h^* (i \in I(1|n))$ be the corresponding
dual basis.

A standard set $\Pi$ of simple roots is
%$\{\alpha_0 :=\delta_{-1} -\delta_1, \alpha_1 :=\delta_1
%-\delta_2, \cdots, \alpha_{n-1} := \delta_{n-1} -\delta_n,
%\alpha_n := 2\delta_n \}$
%
$\{\delta_{-1} -\delta_1, \delta_1 -\delta_2, \cdots, \delta_{n-1}
-\delta_n, 2\delta_n \}$ with $\delta_{-1} -\delta_1$ being odd.
The Dynkin diagram with respect to $\Pi$ is
\vspace{.2 cm}
\begin{equation*}
%Dynkin diagram for osp(2|2n)
\begin{array}{cc}
\setlength{\unitlength}{0.3in}
\begin{picture}(10,2)
\put(0,1){\makebox(0,0)[c]{$\otimes$}}
\put(2.1,1){\makebox(0,0)[c]{$\circ$}}
\put(7.40,1){\makebox(0,0)[c]{$\circ$}}
\put(8.85,1){\makebox(0,0)[c]{$\circ$}}
 \put(0.23,1){\line(1,0){1.7}}
 \put(2.27,1){\line(1,0){1.3}}
 \put(4.0,.96){$\ldots\ldots$}
 \put(5.9,1){\line(1,0){1.3}}
 \put(7.5,0.85){$\Longleftarrow$}
 \put(8.32,0.85){$=$}
\put(0,0.4){\makebox(0,0)[c]{$\delta_{-1} -\delta_{1}$}}
\put(2.4,0.4){\makebox(0,0)[c]{$\delta_{1} -\delta_{2}$}}
\put(7.2,0.4){\makebox(0,0)[c]{$\delta_{n-1} -\delta_{n}$}}
\put(9.05,0.4){\makebox(0,0)[c]{$2\delta_{n}$}}
\end{picture}
\end{array}
\end{equation*}
\vspace{.2 cm}

The set of positive roots is $\Delta^+ =\Delta_0^+ \cup
\Delta_1^+$, where
\begin{align*}
\Delta_0^+ &= \{\delta_i \pm \delta_j, 1\le i
<j \le n\} \cup \{2\delta_i, 1\le i \le n\},  \\
\Delta_1^+  &= \{\delta_{-1} \pm \delta_i, 1\le i \le n\}.
\end{align*}
We define the odd root vectors: for $1\le j \le n$,
\begin{eqnarray}
X_{\delta_{-1} +\delta_j} := E_{-2,j+n} +E_{j,-1}, &&
X_{-\delta_{-1} -\delta_j} :=  E_{j+n,-2} -E_{-1,j},\label{oddvec1C}\\
X_{\delta_{-1} -\delta_j} :=  E_{-2,j} - E_{j+n,-1}, &&
X_{-\delta_{-1}+\delta_j} :=E_{-1,j+n}+E_{j,-2}.\label{oddvec2C}
\end{eqnarray}

\begin{proposition}
The Chevalley basis elements for $sp(2n)$, the vector $E_{-2,-2}
-E_{-1,-1}$ in $so(2)$, together with the odd root vectors
$X_\alpha \,(\alpha \in \Delta_1)$ in
(\ref{oddvec1C}-\ref{oddvec2C}) form a basis for Lie superalgebra
$\spo(2n|2)$ with integer structure constants.
\end{proposition}
The proof of this proposition is again by a direct computation.
This basis will be called the {\em Chevalley basis} for
$\spo(2n|2)$.
\begin{remark}
The Chevalley basis for $\spo(2n|\ell)$ introduced in this paper
has the following remarkable property (which can be verified case
by case):
$$[X_\alpha, X_\beta] = \pm (r+1) X_{\alpha +\beta}, \text{ if }
\alpha, \beta, \alpha +\beta \in \Delta,$$
where $r$ is determined by the $\alpha$-string of roots through
$\beta$: $\{\beta + i\alpha \mid -r \le i \le q \} \subseteq
\Delta \cup \{0\}$. Of course, this is part of the definition of
the Chevalley basis for the even subalgebra of $spo(2n|\ell)$, and
it makes no difference in this case (and most other cases) to have
$\Delta \cup \{0\}$ here instead of the usual $\Delta$.  The
additional $\{0\}$ is needed exactly when $\ell =2m+1$ and $\alpha
= \beta =\pm \delta_i$ for $-n \le i<0$, where twice a root
happen to be a root.

This observation suggests a possible uniform definition of
Chevalley basis for all simple Lie superalgebras of the classical
types (which has been classified \cite{Kac}).
\end{remark}

%%
%%
%%
%\section{Basis for the superalgebra of distributions}
%
%
%
\subsection{Basis of the superalgebra of distributions}

Let $G =\SpO(2n|\ell)$ and $\mathfrak g =\spo(2n|\ell)$ with
Chevalley basis $\{X_\alpha, H_s, \mid \alpha\in \Delta, s\in
\Pi_\ev \}$. Denote by $\mathfrak g_\C$ the complex Lie
superalgebra $\spo(2n|\ell, \C)$. The PBW theorem for Lie
superalgebra implies that the universal enveloping superalgebra
$\mathcal U(\mathfrak g_\C)$ has a basis
$$\prod_{\alpha\in \Delta_0} X_{\alpha}^{n_\alpha} \prod_{s \in
\Pi_\ev}  H_{s}^{m_s} \prod_{\beta\in \Delta_1}
X_{\beta}^{\varepsilon_\beta}
$$
with $n_\alpha, m_s \in \Z_+$ and $\varepsilon_\beta \in \{0,1\}$.
Define the Kostant $\Z$-form $\mathcal U_\Z$ to be the
$\Z$-subalgebra of $\mathcal U(\mathfrak g_\C)$ generated by the
following elements
\begin{align}
&X_{\alpha}^{(r)} :=\frac{X_\alpha^r}{r!},\:\: \alpha \in
\Delta_0, r \in \Z_+; \quad   X_{\beta}, \:\: \beta\in \Delta_1;
 \nonumber\\
&{H_{s} \choose  m_s} :=\prod_{j=0}^{m_s-1} (H_s-j)\slash m_s!,
\:\: s\in \Pi_{\ev}, m_s\in \Z_+.  \nonumber
 \end{align}
We define $\mathcal U_k := \mathcal U_\Z \otimes_\Z k.$

\begin{theorem}\label{KOSTANTZFORM}
 \begin{enumerate}
\item[(a)] The superalgebra $\mathcal U_\Z$ is a free $\Z$-module with
basis
\begin{equation}\label{BASIS}
\prod_{\alpha\in \Delta_0} X_{\alpha}^{(n_\alpha)} \prod_{s \in
\Pi_\ev}  {H_{s} \choose  m_s} \prod_{\beta\in \Delta_1}
X_{\beta}^{\varepsilon_\beta}
\end{equation}
for $n_\alpha, m_s \in \Z_+$ and $\varepsilon_\beta \in \{0,1\}$,
where the product is taken in any fixed order.
\item[(b)] As Hopf superalgebras, $\mathcal U_k$ is isomorphic to
$\Dist(G)$.
\end{enumerate}
\end{theorem}

\begin{proof}
The isomorphism in (b) is the reduction modulo $p$ of the
isomorphism over $\C$ just as for algebraic groups, cf.
\cite[II.1.12]{Jan}; also see \cite{BK2, BKu} for the supergroups
of type $Q$ and $A$.

The proof of (a) basically follows the proof of \cite[Theorem
2]{St2}, with the following additional computation regarding the odd
root vectors. First, we have for $\beta \in \Delta_1$ that
$$X_\beta {H_s \choose t}
%&= {H_s - \beta(H_s)\choose q} X_\beta,
=\sum_{r=0}^t (-1)^{t-r} {\beta(H_s)\choose t-r} {H_s \choose r}
X_\beta$$
where the coefficients on the right-hand side are integers thanks
to a Chevalley basis property $\beta(H_s) \in \Z$.

Given $\alpha \in\Delta_0, \beta\in \Delta_1$, we have $(\ad
X_\alpha)^3 X_\beta =0$ since $\beta + 3 \alpha$ is never a root
for $\spo(2n|\ell)$ by inspection. If $(\ad X_\alpha)^2
X_\beta=0$, then
$$X_\beta X_\alpha^{(r)}
=X_\alpha^{(r)} X_\beta +  \sum_{\nu\in \Delta_1}
c_{\beta,\alpha}^{\nu} X_\alpha^{(r-1)} X_\nu$$
where $[X_\beta, X_\alpha] =\sum_{\nu\in \Delta_1}
c_{\beta,\alpha}^{\nu} X_\nu$ with all $c_{\beta,\alpha}^{\nu}\in
\Z$ by the construction of Chevalley basis for $\spo(2n|\ell)$.

By inspection, the inequality $(\ad X_\alpha)^2 X_\beta \neq 0$
occurs exactly when $\ell =2m+1$, $\alpha =\delta_j$ and $\beta =
\pm \delta_i -\delta_j$  (or the pair of $\alpha, \beta$ with a
simultaneous change of signs), where $-n \le i < 0<j\le m$. By
Remark~\ref{shortvector} and the explicit formulas for odd root
vectors, we have
\begin{align*}
X_\beta X_\alpha^{(r)}
 &=X_\alpha^{(r)} X_\beta +   X_\alpha^{(r-1)} [X_\beta, X_\alpha]
 + X_\alpha^{(r-2)} \cdot \frac12 (\ad X_\alpha)^2 X_\beta \\
 &= X_\alpha^{(r)} X_\beta +  X_\alpha^{(r-1)} X_{\pm \delta_i}
 - X_\alpha^{(r-2)} X_{\pm \delta_i
+\delta_j}.
 \end{align*}
We observe that all the coefficients on the right-hand side are
integers.
\end{proof}
From now on, we will always identify $\Dist(G)$ with $\mathcal
U_k$ (and also the corresponding subalgebras between them).

\subsection{Various subalgebras}
%\begin{remark}
Let $T$ be the maximal torus of $G =\SpO(2n|\ell)$ which consists
of the diagonal matrices and $B$ be the Borel subgroup
corresponding to $\Delta^+$. Then $Dist(T)$ has a basis given by
${H_{s} \choose  m_s}$ for $m_s \in \Z_+, s \in \Pi_\ev$, and
$Dist(B)$ has a basis given by
\begin{eqnarray} \label{eq:basis}
\prod_{\alpha\in \Delta_0^+} X_{\alpha}^{(n_\alpha)} \prod_{s \in
\Pi_\ev}  {H_{s} \choose  m_s} \prod_{\beta\in \Delta_1^+}
X_{\beta}^{\varepsilon_\beta}
\end{eqnarray}
for $n_\alpha, m_s \in \Z_+$ and $\varepsilon_\beta \in \{0,1\}.$

Clearly, the Frobenius morphism $F^r$ preserves the various
subgroups $T, B_\ev,B$ of $G$, for $r \ge 1$. We denote by
$\mathcal U_r$ the $k$-subalgebra of $\mathcal U_k$ spanned by the
elements (\ref{BASIS}) with $0 \le n_\alpha < p^r$, $0\le m_s <
p^r$ and $\varepsilon_\beta \in \{0,1\}$. As in the purely even
case (cf. \cite{Jan}), $\Dist(G_r)\cong \mathcal U_r$, and we will
not distinguish these two superalgebras. In particular,
$\Dist(G_1)$ is the restricted enveloping superalgebra of
$\mathfrak g$. Similarly, a basis for $Dist(B_r)$ of the $r$-th
Frobenius kernel $B_r$ is given by the elements (\ref{eq:basis})
for $0 \le n_\alpha < p^r$, $0\le m_s < p^r$ and
$\varepsilon_\beta \in \{0,1\}$.
%\end{remark}

%%
%%
%%
%%
\section{A tensor product theorem for $\SpO(2n|\ell)$}
\label{sec:tensor}

Let $G =SpO(2n|\ell)$. The equivalence of categories established
in Theorem~\ref{th:equiv} allows us to study $G$-modules via a
highest weight theory of $Dist(G)$-modules. The proofs in this
section are similar to \cite{BK2, BKu, Ku} for $Q(n)$ and
$GL(m|n)$, which in turn were super modification from standard
developments in the algebraic group setup (cf. \cite{CPS, Jan}).
\subsection{Highest weight modules}

We continue to denote $G =SpO(2n|\ell)$ with maximal torus $T$ and
Borel subgroup $B$ as before. The character group is
$$X(T) = \left \{\Sigma_{i \in I(n|m)} \la_i \delta_i \mid \text{ all }
\la_i \in \Z \right \}.
$$
A standard symmetric bilinear form on $X(T)$ is defined by
\begin{eqnarray*}
 (\delta_i, \delta_j) =\left \{
 \begin{array}{rl}
0, & \text{ if } i \neq j\\
1, & \text{ if } i=j<0\\
-1, & \text{ if } i=j>0.
\end{array}
\right.
\end{eqnarray*}
Denote $Y(T) = \Hom (\mathbb G_m, T)$. There is a natural pairing
$\langle\;\;, \;\; \rangle: X(T) \times Y(T) \rightarrow \Z$. Given
an even root $\alpha \in \Delta_0$, its coroot $\alpha^\vee \in
Y(T)$ is defined as in \cite{Jan}. The Weyl group $W$ is generated
by the reflections $s_\alpha$ for $\alpha \in \Delta_0$, where
$s_\alpha \la = \la - \langle \la, \alpha^\vee \rangle \alpha$ for
$\la\in X(T)$. The set
$$
X^+(T) := \left \{ \la  =\Sigma_{i \in I(n|m)} \la_i \delta_i \in
X(T) \mid \langle \la, \alpha^\vee \rangle \ge 0 \text{ for all }
\alpha \in \Delta_0^+ \right \}
$$
can be easily identified with
$$\{ \la  \in X(T) \mid
  \la_{-n} \ge \cdots \ge \la_{-1} \ge 0,
 \la_{1} \ge \cdots \ge \la_{m} \ge 0 \}, \text{ if } \ell =2m+1,$$
and identified with
$$\{ \la  \in X(T) \mid
  \la_{-n} \ge \cdots \ge \la_{-1} \ge 0,
 \la_{1} \ge \cdots \ge \la_{m-1} \ge |\la_{m}| \ge 0 \}, \text{ if } \ell =2m.$$

For an $\lambda \in X(T)$ and a $\Dist(G)$-module $M$, the
$\lambda$-weight subspace of $M$ is
$$
M_\lambda=\left\{ m\in M \mid {H_s\choose m_s}m={\langle \lambda,
H_s\rangle \choose m_s}m, \; \forall s\in \Pi_\ev, m_s\in \Z_+
\right\}.
$$
For $ \la  \in X(T)$, we denote the Verma module
$$
M(\la) =Dist(G) \otimes_{Dist(B)} k_\la
$$
where $k_\la$ is the one-dimensional $Dist(B)$-module of weight
$\la$ (in degree $\bar{0}$). It is standard to see that the
$Dist(G)$-module $M(\la)$ has a unique simple quotient $L(\la)$
and that  the $Dist(G)$-modules $L(\la)$, $\la \in X(T)$, are
pairwise non-isomorphic. By definition, $L(\la)$ is $X(T)$-graded
and thus a $T$-module.

\begin{lemma}  \label{lem:simple}
Every simple module $M$ in the category
$(Dist(G),T)$-$\mathfrak{mod}$ is isomorphic to a
finite-dimensional highest weight module $L(\la)$ for some $\la
\in X^+(T)$.
\end{lemma}
\begin{proof}
Since $M$ is simple and locally finite, it is finite-dimensional.
By weight consideration, there exists a highest weight vector
$v_\la$ of weight $\la \in X(T)$ such that $X_\alpha^{(r)} v_\la
=0$ for all $\alpha \in \Delta^+, r \ge 1$, (and $r=1$ for
$\alpha$ odd). It follows that $M =L(\la)$. Now
$Dist(G_\ev).v_\la$ as a subspace of $L(\la)$ is finite
dimensional. A classical result applied to $G_\ev$ asserts that
$\la \in X^+(T)$.
\end{proof}

A major challenge here, which is a super phenomenon and does not
occur for $GL(m|n)$, is that $L(\la)$ for various $\la\in X^+(T)$
may fail to be finite-dimensional in general. By the equivalence
of categories $G$-$\mathfrak{mod}$ $\cong$
$(Dist(G),T)$-$\mathfrak{mod}$ in Theorem~\ref{th:equiv}, we will
no longer distinguish a $G$-module $M$ from a $(Dist(G),T)$-module
$M$. Define
\begin{eqnarray} \label{eq:dag}
X^\dag(T) =\{\la\in X^+(T) \mid L(\la) \text{ is
finite-dimensional}\}.
\end{eqnarray}
By Lemma~\ref{lem:simple}, the $L(\la)$'s, where $\la$ runs over
$X^\dag(T)$, form a complete list of pairwise non-isomorphic
simple $G$-modules. Thus, the classification of simple $G$-modules
boils down to the nontrivial problem of determining explicitly the
set $X^\dag(T)$, which will be solved completely in
Section~\ref{sec:classif}.
\subsection{The $G_r$-modules}

Given $\la \in X(T)$, define the baby Verma module (which is also
a $G_r$-module)
$$Z_r(\la) : = Dist(G_r) \otimes_{Dist(B_r)} k_\la.
$$
One can show as usual that $Z_r(\la)$ has a unique simple
$G_r$-quotient, which will be denoted by $L_r(\la)$. The following
is standard.

\begin{proposition}
Every simple $G_r$-module is isomorphic to $L_r(\la)$ for some
$\la \in X(T)$. Furthermore, $L_r(\la) \cong L_r(\mu)$ if and only
if $\la-\mu \in p^r X(T)$.
\end{proposition}

\subsection{A tensor product theorem}

\begin{proposition} \label{semisimple}
Let $r \ge 1$.
\begin{enumerate}
\item Every simple $G$-module regarded as a $G_r$-module is
completely reducible.
\item For $\la \in X^\dag (T)$ and $0 \neq v_\la \in L(\la)_\la$,
$Dist(G_r).v_\la$ is a $G_r$-submodule of $L(\la)$ and it is
isomorphic to $L_r(\la)$.
\end{enumerate}
\end{proposition}

\begin{proof}
Let $M$ be a simple $G$-module. Take any simple $G_r$-submodule
$M_1$ in $M$. Then, $\sum_{g \in G_\ev (k)} gM_1$ is a completely
reducible $G_r$-module; it is also a $G_\ev$-module and a
$G_1$-module, hence a $G$-module by Lemma~\ref{lem:keygen}. Thus,
$M=\sum_{g \in G_\ev (k)} gM_1$.

From the $B_r$-homomorphism $k_\la \rightarrow L(\la)$ and
Frobenius reciprocity, we obtain a $G_r$-homomorphism $Z_r(\la)
\rightarrow L(\la)$ with image $Dist(G_r).v_\la$. Now (2) follows
from that $Dist(G_r).v_\la$ is completely reducible by (1) and
that $Z_r(\la)$ has a simple $G_r$-quotient $L_r(\la)$.
\end{proof}

Denote $X_r(T) =\{\la \in X(T) \mid 0 \le \langle \la, \alpha^\vee
\rangle < p^r$ \text{ for all } $\alpha \in \Pi_\ev \}$.

\begin{proposition} \label{Frobsimple}
 For $\la \in X_r(T) \cap X^\dag(T)$, the restriction to
$G_r$ of the simple $G$-module $L(\la)$ remains to be simple, and
is isomorphic to $L_r(\la)$.
\end{proposition}

\begin{proof}
By Proposition~\ref{semisimple}, the $G_r$-module
$M:=Dist(G_r).v_\la$ is isomorphic to $L_r(\la)$. It remains to
show that $M$ is a $G$-module, or to show that $M$ is a
$Dist(G)$-module. Since $Dist(G)$ is generated by $Dist(G_\ev)$
and $Dist(G_1)$, it suffices to show that $M$ is preserved by the
action of $Dist(G_\ev)$. Since $M$ is a $G_r$-module and $B_\ev$
normalizes $G_r$, $M$ is already preserved by the action of
$B_\ev$ or $Dist(B_\ev)$. So it remains to show that $M$ is
preserved by the action of $X_\alpha^{(m)}$ for every $m\ge 1$ and
negative even roots $\alpha$. This can be proved by the inductive
argument used for the supergroup $GL(m|n)$ in \cite[Lemma~6.3]{Ku}
(which goes back to Borel according to \cite[Lemma~9.8]{BK2} for
$Q(n)$). We omit the details.
\end{proof}

Let us denote by $L_\ev(\la)$ the simple $G_\ev$-module of highest
weight $\la \in X^+(T)$. The pullback $L_\ev(\la)^{[r]}
:=(F^r)^*L_\ev(\la)$ as a $G$-module is clearly simple and
isomorphic to $L(p^r\la)$. We have the following analogue of the
Steinberg tensor product theorem \cite{St1}, which will become
more precise with the determination of the set $X^\dag(T)$ in the
next section.

\begin{theorem} \label{Steinberg}
\begin{enumerate}
\item  For $r \ge 1$, $\la \in X_r(T) \cap X^\dag(T)$, and $\mu
\in X^+(T)$, we have
$$L(\la +p^r \mu) \cong L(\la) \otimes L_\ev(\mu)^{[r]}.$$
\item Let $\la =\sum_{i=0}^m p^i \la^{(i)}$ be such that
$\la^{(0)} \in X_1(T) \cap X^\dag(T)$ and $\la^{(i)}  \in X_1(T)$
for $1\le i \le m$. Then
$$L(\la) \cong L(\la^{(0)}) \otimes L_\ev(\la^{(1)})^{[1]}
\otimes \cdots \otimes L_\ev(\la^{(m)})^{[m]}.$$
\end{enumerate}
\end{theorem}

\begin{proof}
Both statements follow easily with the help of the classical
Steinberg tensor product theorem for $G_\ev$, once we establish
the special case of (1) with $r=1$. This case can be proved using
Propositions~\ref{semisimple} and \ref{Frobsimple} in the same way
as for algebraic groups \cite{Jan} (cf. \cite{BK2, Ku}).
\end{proof}

\section{Classification of simple $SpO$-modules}
\label{sec:classif}
\subsection{The cases of $SpO(2n|1)$ and $SpO(2n|2)$}
Among all $spo(2n|\ell)$ with $n>0$ and $\ell>0$, $spo(2n|1)$ and
$spo(2n|2)$ distinguish themselves from the others in that they
are three-component $\Z$-graded Lie superalgebra: $\mathfrak g=
\mathfrak g_{-1} + \mathfrak g_0 +\mathfrak g_{+1}$, where
$\mathfrak g_0$ coincides with the even subalgebra $\mathfrak
g_{\bar{0}}$, and $\mathfrak g_{\pm 1}$ is generated by root
vectors associated to positive/negative odd roots.
\begin{proposition}
Let $n>0$ and $\ell=1$ or $2$. Then the $L(\la)$, where $\la$ runs
over $X^+(T)$, form a complete list of pairwise non-isomorphic
simple $\SpO(2n|\ell)$-modules.
\end{proposition}

\begin{proof}
The proof is the same as in the case of characteristic zero (cf.
\cite{Kac}). Let $G =\SpO(2n|\ell)$. By Lemma~\ref{lem:simple},
every simple $G$-module is of the form $L(\la)$ for some $\la \in
X^+(T)$. Because of the three-component $\Z$-grading on $\mathfrak
g =\spo(2n|\ell)$, one has a decomposition $Dist(G)
=Dist(G)_{-1}\cdot Dist(G_\ev)\cdot Dist(G)_{+1}$, where
$Dist(G)_{\pm 1}$ is generated by the odd positive/negative root
vectors and $Dist(G_\ev)\cdot Dist(G)_{+1}$ is a subalgebra. Then
for every $\la \in X^+(T)$, the irreducible $Dist(G_\ev)$-module
$L_\ev(\la)$ extends trivially to a $Dist(G_\ev)\cdot
Dist(G)_{+1}$-module. Hence, $K(\la) := Dist(G)
\otimes_{Dist(G_\ev)\cdot Dist(G)_{+1}} L_\ev(\la)$ is
finite-dimensional of highest weight $\la$, with $L(\la)$ as its
irreducible quotient.
\end{proof}

\subsection{Combinatorics related to Mullineux conjecture}

Let $\mu =(\mu_1, \mu_2, \ldots)$ be a partition. We will identify
it with its Young diagram and denote by $\ell(\mu)$ its length. The
{\em rim} of the Young diagram $\mu$ is the set of cells $(i,j)$
such that the cell $(i+1, j+1)$ is not in $\mu$. The {\em $p$-rim}
is a subset of the rim defined as follows in terms of $p$-segments.
The first $p$-segment consists of the first $p$ cells in the rim
from the left. The second $p$-segment starts with the first cell in
the rim strictly to the right of the previous segment, and so on.
The last $p$-segment is allowed to consist of possibly cells fewer
than $p$.
%The cardinality of the $p$-rim of $\mu$ is denoted by $a(\mu)$.

A cell of $\mu$ is {\em $p$-removable} if it is at the end of a
row of $\mu$ and it is in the $p$-rim but not at the end of any
$p$-segment. Denote by $J(\mu)$ the partition obtained from $\mu$
by deleting all $p$-removable cells of $\mu$. The number, denoted
by $j(\mu)$, of all $p$-removable cells in $\mu$ is then given by
$$ j(\mu) =|\mu| -|J(\mu)|.
$$
Let us denote $J^i(\mu) = J(J^{i-1}(\mu))$ for $i\ge 1$, with
$J^0(\mu) =\mu$.
%It is known (cf. e.g. \cite{BKu}) that
%%
%\[
% j(\la) =
%  \left \{
% \begin{array}{ll}
%  a(\la) -\la_1 & \text{ if } a(\la) \equiv 0  \mod p, \\
%  a(\la) -\la_1 +1 & \text{ if } a(\la) \not \equiv 0  \mod p.
%  \end{array}
% \right.
%\]
A partition $\mu =(\mu_1, \mu_2, \ldots)$ is {\em $p$-restricted}
(or simply {\em restricted}) if either $p=0$ or $p>0$ and $\mu_i
-\mu_{i+1} <p$ for all $i \ge 1$. Note that the notions $j$ and
$J$ make sense for arbitrary partitions, and for any partitions
$\mu$ and $\nu$,
\begin{eqnarray} \label{eq:j}
 j(\mu + p \nu) =j (\mu).
\end{eqnarray}
Let $\mathcal{RP} (d)$ be the set of restricted partition of $d$,
and let $\mathcal{RP} = \bigsqcup_{d \ge 0} \mathcal{RP} (d)$.

%%%%
%%%%
%For example, let $\mu =(5,4,3,3,1,1)$ and $p=5$. The $p$-rim of
%$\mu$ has two $p$-segments and it is consisted of the cells marked
%by $\ast$ or by $\circledast$, where the $p$-removable cells are
%those marked by $\circledast$. Hence $j(\mu) = 4$, and by deleting
%the cells marked by $\circledast$ from $\mu$, we obtain that
%$J(\mu) =(4,3,3,3)$.
%%
%$$\mu =
%\young(\,\,\,\ast\circledast,\,\,\,\circledast,\,\,\,,\ast\ast\ast,\circledast,\circledast)$$
%%%%
%%%%
%%%%

For example, let $\mu =(5,4,3,3,1,1)$ and $p=5$. The $p$-rim of
$\mu$ consists of two $p$-segments and it is consisted of the
cells colored in black (including both $\bullet$ and the
double-circled cells) as follows, where the $p$-removable cells
are the double-circled ones. Hence $j(\mu) = 4$,  and by deleting
the double-circled cells from $\mu$ we obtain that $J(\mu)
=(4,3,3,3)$.
%$$
%\setlength{\unitlength}{0.007500in}%
%\begin{picture}(108,124)(139,514)
%\multiput(143,618)(20,0){3}{\circle{5}}
%\multiput(143,598)(20,0){3}{\circle{5}}
%\multiput(143,578)(20,0){3}{\circle{5}}
%\multiput(203,618)(20,0){2}{\circle*{6}}
%\multiput(203,598)(20,0){1}{\circle*{6}}
%\multiput(143,558)(20,0){3}{\circle*{6}}
%\multiput(143,538)(20,0){1}{\circle*{6}}
%\multiput(143,518)(20,0){1}{\circle*{6}}
%\end{picture}
%$$
%$J(\mu)$ is obtained by deleting the double-circled cells:
$$
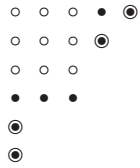
\setlength{\unitlength}{0.007500in}%
\begin{picture}(108,124)(139,514)
\multiput(143,618)(20,0){3}{\circle{5}}
\multiput(143,598)(20,0){3}{\circle{5}}
\multiput(143,578)(20,0){3}{\circle{5}}
\multiput(203,618)(20,0){2}{\circle*{6}}
\multiput(203,598)(20,0){1}{\circle*{6}}
\multiput(143,558)(20,0){3}{\circle*{6}}
\multiput(143,538)(20,0){1}{\circle*{6}}
\multiput(143,518)(20,0){1}{\circle*{6}}
\multiput(143,518)(20,0){1}{\circle{10}}
\multiput(203,598)(20,0){1}{\circle{10}}
\multiput(223,618)(20,0){1}{\circle{10}}
\multiput(143,538)(20,0){1}{\circle{10}}
\end{picture}
$$

\begin{theorem} \label{mull}
For $\mu \in \mathcal{RP}$, set $\la_i =j(J^{i-1}(\mu))$ for $i\ge
1$. Then, sending $\mu \in \mathcal{RP}$ to $\texttt{M} (\mu)
:=(\la_1, \la_2, \cdots)$ defines a bijection $\texttt M:
\mathcal{RP} \rightarrow \mathcal{RP}$ which satisfies $\texttt{M}
=\texttt{M}^{-1}$.
\end{theorem}
It is well known that the simple $kS_d$-modules $D_\mu$ are
parameterized by $\mu\in \mathcal{RP} (d)$, which are the
transposes of the $p$-regular partitions. The Mullineux conjecture
\cite{Mu} (now a theorem due to \cite{FK}) states that $D_\mu
\otimes \text{sgn} \cong D_{\texttt{M}(\mu)}$, where $\text{sgn}$
is the one-dimensional sign $kS_n$-module. The formulation in
Theorem~\ref{mull} due to \cite{Xu} of the {\em Mullineux
bijection} $\texttt M: \mathcal{RP} (d) \rightarrow \mathcal{RP}
(d)$ is equivalent to the earlier formulations by Mullineux,
Kleshchev and others (cf. \cite{BKu} for references and history).

\subsection{The classification}
Recall $T$ is the maximal torus of $\SpO(2n|\ell)$, with $\ell =2m$
or $2m+1$. A weight $\la \in X(T)$ can be expressed as $\la =\sum_{i
\in I(n|m)} \la_i \delta_i$.

Recall that the simple $G$-modules have been classified for $G
=\SpO(2n|1)$ and $\SpO(2n|2)$. (The case of $\SpO(2n|1)$ also fits
into the general statement below.)

\begin{theorem} \label{th:classif}
\begin{enumerate}
\item Let $\,n \ge 1$. A complete list of pairwise non-isomorphic
simple $SpO(2n|2m+1)$-modules is $\{L(\la)\}$, where $\la$ runs
over the set:
 \begin{eqnarray*}
 \left \{
 \Sigma_{i \in I(n|m)} \la_i \delta_i\mid \la_{-n} \ge \cdots \ge \la_{-1} \ge 0,
 \la_{1} \ge \cdots \ge \la_{m} \ge 0,
 \right. \\
 \left.
  \text{all } \la_i \in \Z, \;
 j(\la_{1}, \ldots, \la_{m}) \le \la_{-1} \right  \}.
 \end{eqnarray*}
 \item Let $m \ge 2$ and $n \ge 1$.
A complete list of pairwise non-isomorphic simple
$SpO(2n|2m)$-modules is $\{L(\la)\}$, where $\la$ runs over the
following set:
 \begin{eqnarray*}
 \left \{
 \Sigma_{i \in I(n|m)} \la_i \delta_i\mid \la_{-n} \ge \cdots \ge \la_{-1} \ge 0,
 \la_{1} \ge \cdots \ge \la_{m-1} \ge |\la_{m}| \ge 0,
 \right. \\
 \left.
  \text{all } \la_i \in \Z, \;
 j(\la_{1}, \ldots, |\la_{m}|) \le \la_{-1} \right \}.
 \end{eqnarray*}
 \end{enumerate}
\end{theorem}

\begin{remark}
Theorem~\ref{th:classif} and its proof below using odd reflections
remain to be valid over an algebraically closed field of
characteristic $0$, where $j(\mu)$ (resp. $\texttt M (\mu)$) for a
partition $\mu$ is simply replaced by $\ell(\mu)$ (resp. the
conjugate partition $\mu^t$). In this case, a classification of
the finite-dimensional simple $spo(2n|\ell)$-modules using the
Dynkin labels appeared in \cite{Kac}, whose sketchy proof uses
totally different ideas and does not apply to the modular case. It
can be shown that the statement in {\em loc. cit.} is equivalent
to ours in characteristic zero.

The idea of using odd reflections for determining dominant weights
(in characteristic zero) goes back at least to \cite{LSS} (also
see \cite{PS} for a general setup).
\end{remark}

\begin{remark}
The simple {\em polynomial} $GL(n|m)$-modules (i.e. the
composition factors appearing in various tensor powers of the
natural $GL(n|m)$-module) were classified by Brundan-Kujawa
\cite{BKu} (this generalizes some earlier partial result of
Donkin). According to Theorem~\ref{th:classif}, the simple
$SpO(2n|2m+1)$-modules happen to admit the same parameterizing set
as the simple polynomial $GL(n|m)$-modules. This remarkable fact
suggests that the inclusion $GL(n|m) \leq SpO(2n|2m+1)$ could be
significant for further development.
\end{remark}

We will present a detailed proof below in the case of
$G=SpO(2n|2m+1)$. Denoting the explicit set of weights in part~(1)
of the above theorem by ${\mathcal X}^\dag(T)$ and recalling
(\ref{eq:dag}), we shall prove
$$X^\dag(T) ={\mathcal X}^\dag(T).$$

\subsection{The inclusion ${\mathcal X}^\dag(T) \subseteq X^\dag(T)$}

Let $\mathcal G =SpO(2N|2M+1)$ with $N \ge n$ and $M \ge m$.
%All systems of simple roots for $\mathcal G
%=SpO(2N|2M+1)$ are not conjugate under the Weyl group action.
We denote by $\delta_i$  with $i \in I(N|M)$ the standard weights
for $\mathcal G$, and by $T_{N,M}$ the maximal torus for $\mathcal
G$ to distinguish from $T$ for $G =SpO(2n|2m+1)$. Denote by
$S_{N+M}$ the symmetric group on the set $I(N|M)$ with subgroups
$S_N$ and $S_M$ whose notations are self-explained. According to
\cite{Kac}, up to the Weyl group equivalence, the systems of
simple roots $\Pi^\sigma$ are determined by the following set of
minimal length coset representatives in $S_{N+M} /S_N \times S_M$:
 $$D_{N,M} := \{\sigma \in S_{N+M} \mid \sigma^{-1} (-N) < \cdots < \sigma^{-1}(-1),
 \sigma^{-1} (1) < \cdots <\sigma^{-1} (M)\}. $$
 More precisely, if we denote $\sigma(\delta_i) =\delta_{\sigma(i)}$ and $\sigma(\delta_i
\pm \delta_j) =\delta_{\sigma(i)} \pm \delta_{\sigma(j)}$, etc, then
$ \Pi^\sigma := \{ \sigma (x) \mid x \in \Pi  \}.$
Denote by $\Delta^{\sigma+}$ the set of positive roots relative to
$\Pi^\sigma$. We shall denote by $L^\sigma(\la)$ the irreducible
highest weight $Dist(\mathcal G)$-module of highest weight $\la$
relative to $\Pi^\sigma$. In particular, $L^1(\la) =L(\la)$.

\begin{lemma}  \label{lem:weight}
A weight $\la$ of the form $\la =\sum_{i=-N}^{-1} \la_i \delta_i$,
where $(\la_{-N}, \ldots, \la_{-1})$ forms a partition, always
belongs to $X^\dag (T_{N,M})$.
\end{lemma}

\begin{proof}
For $1 \le r \le N$, the $r$-th exterior power $\Lambda^r V$ of
the natural $\mathcal G$-module $V =k^{2N|2M+1}$ is of highest
weight $\Lambda_r := \sum_{i=-N}^{r-N-1} \delta_i$, and thus
$\Lambda_r \in X^\dag (T_{N,M}).$ Denote the conjugate partition
of $(\la_{-N}, \ldots, \la_{-1})$ by $(r_1, \ldots, r_s)$, with $N
\ge r_1 \ge \cdots \ge r_s \ge 0$. Clearly $\la =\Lambda_{r_1}
+\cdots +\Lambda_{r_s}$, and thus $L(\Lambda_{r_1}) \otimes \cdots
\otimes L(\Lambda_{r_s})$ contains $L(\la)$ as a quotient
$\mathcal G$-module.
\end{proof}

Each $\Pi^\sigma$ gives rise to a Borel subgroup $B^\sigma$ of
$\mathcal G$. The Verma module $M^\sigma(\la) := \Dist(\mathcal
G)\otimes_{\Dist(B^\sigma)} k_\la$ has a unique irreducible
quotient $\Dist(\mathcal G)$-module $L^\sigma(\la)$. The method of
{\em odd reflections} as presented in the next lemma is an
effective tool of relating different conjugacy classes of Borel
subgroups. It was exploited earlier by Serganova and many other
authors in characteristic zero and then formulated in
\cite[Lemma~4.2]{BKu} for $GL(m|n)$ in characteristic $p>2$. The
following argument is adapted from {\em loc. cit.} in terms of
roots and root vectors.

\begin{lemma}  \label{lem:oddref}
Let $\la \in X(T_{N,M})$, and let $\alpha =\pm \delta_i \pm
\delta_j$ be an odd root. Suppose that $\sigma, \sigma' \in
D_{N,M}$ are such that
%$\sigma =(i,j) \sigma'$ and
$\Delta^{\sigma' +} =\Delta^{\sigma+} \cup \{-\alpha\} \backslash
\{\alpha \}$. Then,
\[
L^\sigma (\la) \cong \left \{
 \begin{array}{ll}
  L^{\sigma'}(\la) & \text{ if } (\la, \alpha) \equiv 0  \mod p, \\
  L^{\sigma'}(\la -\alpha) & \text{ if } (\la, \alpha) \not \equiv 0  \mod
  p.
 \end{array}
 \right.
\]
\end{lemma}

\begin{proof}
Let $v$ be a $\sigma$-highest weight vector in $L^\sigma (\la)$ of
weight $\la$. For any $\beta \in \Delta^{\sigma+} \cap
\Delta^{\sigma'+}$, we have
\begin{eqnarray}  \label{eq:zero}
X_\beta X_{-\alpha} v = [X_\beta, X_{-\alpha}] v=0
\end{eqnarray}
since either $\beta -\alpha$ is not a root or it belongs to
$\Delta^{\sigma+} \cap \Delta^{\sigma'+}$.

Since $2\alpha$ is not a root and $\alpha$ is odd, $X_{-\alpha}^2
v =0$.

If $X_{-\alpha} v =0$, then $v$ is a $\sigma'$-highest weight
vector of weight $\la $, and $L^\sigma (\la) \cong L^{\sigma'}
(\la)$.

If $X_{-\alpha} v \neq 0$, then $X_{-\alpha} v$ is a
$\sigma'$-highest weight vector of weight $\la -\alpha$, and
$L^\sigma (\la) \cong L^{\sigma'} (\la -\alpha)$.

Now, $X_{-\alpha} v \neq 0$ if and only if there exists $b \in
Dist (B^\sigma)$ such that $bX_{-\alpha} v =v$. By
(\ref{eq:zero}), we only need to consider $b$ to be a nonzero
scalar multiple of $X_\alpha$. Finally $X_\alpha X_{-\alpha} v =
(\la, \alpha)v.$ This completes the proof.
\end{proof}

Let ${\bf w} \in D_{N,M}$ be the distinguished permutation on
$I(N|M)$ such that the sequence $\{{\bf w}(i)\}_{i \in I(N|M)}$ is
the sequence which starts first from $-N$ to $n-N-1$ increasingly,
then from $1$ to $M$, and finally from $n-N$ to $-1$ increasingly.
Recall the Mullineux map $\texttt M$ from Proposition~\ref{mull}.

\begin{proposition}  \label{weight}
Let $N \ge n$ and $(\mu_{-N}, \mu_{1-N}, \ldots, \mu_{-1})$ be a
partition of length $\le N$ such that the `tail' $\mu^{>n}
:=(\mu_{n-N}, \mu_{n-N+1}, \ldots, \mu_{-1})$ is restricted. Let
$M \ge m$ be such that $M \ge \ell (\texttt{M} (\mu^{>n}))$. Set
$\mu =\sum_{i=-N}^{-1} \mu_i \delta_i$. Then we have an
isomorphism of $\mathcal G$-modules  $L(\mu) \cong L^{\bf
w}(\mu^{\bf w})$ with
$$\mu^{\bf w} := \sum_{i=-N}^{n-N-1} \mu_i \delta_i +
\sum_{i=1}^{M} \texttt{M}(\mu^{>n})_i \delta_{i}$$
where $\texttt{M}(\mu^{>n})_i$ is the $i$-th part of the partition
$\texttt{M}(\mu^{>n}) :=(\texttt{M}(\mu^{>n})_1,
\texttt{M}(\mu^{>n})_2, \cdots).$
\end{proposition}

\begin{proof}
We will apply Lemma~\ref{lem:oddref} repeatedly with an ordered
sequence of odd reflections associated to the following $M(N-n)$
odd roots:
\begin{eqnarray}
\underbrace{\delta_{-1} -\delta_{1}, \delta_{-2} -\delta_{1},
\cdots, \delta_{n-N} -\delta_{1}};\;
\underbrace{\delta_{-1} -\delta_{2}, \delta_{-2} -\delta_{2},
\cdots, \delta_{n-N} -\delta_{2}};\;
\cdots;   \nonumber \\
\underbrace{\delta_{-1} -\delta_{M}, \delta_{-2} -\delta_{M},
\cdots, \delta_{n-N} -\delta_{M}}. \label{eq:oddroots}
\end{eqnarray}
In this way, $\Delta^+$ is replaced by $\Delta^{{\bf w}+}$. After
the first $N-n$ steps of odd reflections, the weight $\mu$ is
replaced by
$$\tilde{\mu} :=\sum_{i=-N}^{-1} \mu_i \delta_i -\sum_{i=n-N}^{-1}
x_i \delta_i + \left(\sum_{i=n-N}^{-1} x_i \right) \delta_{1}.
$$
Here the numbers $x_i \in \{0,1\}$ with $n-N \le i \le -1$ are
defined starting from $i=-1$ by the following formula:
\[
x_i = \left \{
 \begin{array}{ll}
  0 & \text{ if } \mu_i + \sum_{a= i+1}^{-1} x_{a} \equiv 0  \mod p, \\
  1 & \text{ if } \mu_i + \sum_{a= i+1}^{-1} x_{a} \not\equiv 0  \mod p.
 \end{array}
 \right.
\]
By \cite[Lemma~6.2]{BKu}, one has
$$\mu_i -x_i =J(\mu^{>n})_i, \text{ for } n-N \le i \le -1,
$$
and thus
$$
\tilde{\mu} =\sum_{i=-N}^{n-N-1} \mu_i \delta_i +\sum_{i=n-N}^{-1}
J(\mu^{>n})_i \delta_i + j(\mu^{>n}) \delta_{1}.$$
Here and below we write for $r \ge 1$ that
$ J^r(\mu^{>n}) =(J^r(\mu^{>n})_{n-N}, \cdots,
J^r(\mu^{>n})_{-1}).$

Repeating the above reasoning for the next $N-n$ odd reflections
among (\ref{eq:oddroots}), we replace $\tilde{\mu}$ by the weight
$$\sum_{i=-N}^{n-N-1} \mu_i \delta_i +\sum_{i=n-N}^{-1}
J^2(\mu^{>n})_i \delta_i + j(\mu^{>n}) \delta_{1} + j(J(\mu^{>n}))
\delta_{2}
$$
and so on. Recall the definition of the Mullineux map $\texttt M$
and note that $J^M (\mu^{>n}) =\emptyset$ thanks to $M \geq \ell
(\texttt{M} (\mu^{>n}))$. Finally, after all the $M(N-n)$ odd
reflections, we end up with the weight $\mu^{\bf w}$ in the
proposition.
\end{proof}

Now we are ready to prove that ${\mathcal X}^\dag(T) \subseteq
X^\dag(T)$ for $SpO(2n|2m +1)$. First, take $\nu =\sum_{i\in
I(n|m)} \nu_i \delta_i \in {\mathcal X}^\dag (T)$ with the
additional assumption that $\nu^{+} :=(\nu_1, \ldots, \nu_m)$ is
restricted. Denote by $\texttt{M} (\nu^{+}) :=(a_1, a_2, \ldots)$
which is also a restricted partition. It follows by the definition
of ${\mathcal X}^\dag (T)$ that $(\nu_{-n}, \ldots, \nu_{-1}, a_1,
a_2, \ldots)$ is a partition. Choose $M=m$ and $N \ge n
+\ell(\texttt{M} (\nu^{+}))$. By Lemma~\ref{lem:weight}, the
weight
$$\mu :=\sum_{i=-N}^{n-N-1} \nu_{i+N-n} \delta_i +
\sum_{i=n-N}^{-1} a_{i+N-n+1} \delta_i
$$
lies in $X^\dag (T_{N,M}).$ Note that $\bf w$ has been chosen such
that the standard set of positive roots for $SpO(2n|2m+1)$ is a
subset of positive roots relative to $\pi^{\bf w}$ for $\mathcal G
=SpO(2N|2M+1)$. This inclusion of roots is compatible with the
non-standard inclusion $I(n|m) \hookrightarrow I(N|M)$ given by
$(0>)i \mapsto i+n-N, (0<)j \mapsto j$, and it gives rises to an
embedding of supergroups $SpO(2n|2m+1) \leq \mathcal G$. Applying
Proposition~\ref{weight} and $\texttt{M}^{-1} =\texttt{M}$, we
have $\mu^{\bf w}|_T =\nu$. We conclude that $\nu \in X^\dag(T)$
by restricting the $\mathcal G$-module $L^{\bf w} (\mu^{\bf w})$
to the subgroup $SpO(2n|2m+1)$.

Write an arbitrary $\la \in {\mathcal X}^\dag (T)$ (uniquely) as
$\la =\nu + p \sum_{i=1}^{m} \tau_i \delta_{i}$ for some partition
$\tau =(\tau_1, \ldots, \tau_m)$ and $\nu \in {\mathcal X}^\dag
(T)$ with $\nu^{+}$ being restricted. Since $j(\nu^+) =j(\la^+)$
by (\ref{eq:j}), we have $\nu \in {\mathcal X}^\dag (T)$. By the
tensor product Theorem~\ref{Steinberg},
$$L(\la) \cong L(\nu) \otimes L_\ev(\sum_{i=1}^{m} \tau_i
\delta_{i})^{[1]}, $$
whence $\la \in X^\dag(T)$ and ${\mathcal X}^\dag(T) \subseteq
X^\dag(T)$.

\subsection{The inverse inclusion  $X^\dag(T) \subseteq {\mathcal X}^\dag(T)$}

Part of the necessary conditions for $\la \in X^\dag(T)$ is easy
to determine.
\begin{lemma} \label{easypart}
Assume $\la =\sum_{i \in I(n|m)} \la_{i} \delta_i \in X^\dag(T)$.
Then, $(\la_{-n}, \ldots, \la_{-1})$ and $\la^{+} :=(\la_{1},
\ldots, \la_{m})$ are partitions.
\end{lemma}

\begin{proof}
Since $\Delta^+$ contains $\Delta^+(Sp(2n))$, the fact that $\la$
belongs to $X^\dag(T)$ implies that $(\la_{-n}, \ldots, \la_1)$ is
a partition by the classical result (cf. e.g. \cite{Jan}).
Similarly, Since $\Delta^+$ contains $\Delta^+(SO(2m+1))$, $\la^+$
is also a partition.
\end{proof}

To prove $X^\dag(T) \subseteq {\mathcal X}^\dag(T)$, it remains to
show that $j(\la^+) \le \la_{-1}$ for $\la \in X^\dag(T)$. To that
end, we first consider the special case of $SpO(2|2m+1)$ (i.e.
$n=1$) whose maximal torus will be denoted by $T_{1,m}$. Recall
the standard set of simple roots for $SpO(2|2m+1)$ is $\Pi =
\{\delta_{-1} -\delta_1, \delta_1 -\delta_2, \ldots, \delta_{m-1}
-\delta_{m} , \delta_{m}\}$, with $\delta_{-1} -\delta_1$ being
odd. Let ${\bf z} \in S_{1+m}$ be the permutation defined by ${\bf
z}(-1) =1, {\bf z}(1) =2, \cdots, {\bf z}(m-1) =m, {\bf z}(m)
=-1$. Then
$$\Pi^{\bf z} =\{\delta_{1} -\delta_2, \ldots, \delta_{m-1}
-\delta_{m}, \delta_{m} -\delta_{-1}, \delta_{-1}\},$$
with $\delta_{m} -\delta_{-1}$ and $\delta_{-1}$ being odd.

Via an ordered sequence of odd reflections associated to the $m$
odd roots
$$\delta_{-1} -\delta_{1}, \delta_{-1} -\delta_{2}, \ldots, \delta_{-1} -\delta_{m},
$$
the highest weight $\la$ relative to the standard Borel is
replaced by some $\la^{\bf z} =\sum_{i \in I(1|m)} \la^{\bf z}_{i}
\delta_i$, with all $\la^{\bf z}_i \in \Z$, which is a highest
weight relative to the Borel $B^{\bf z}$. That is, $L(\la) \cong
L^{\bf z}(\la^{\bf z})$. Note that $\la^{{\bf z}+} :=(\la^{\bf
z}_1, \ldots, \la^{\bf z}_{m})$ is a partition by considering the
restriction of $L^{\bf z}(\la^{\bf z})$ to the subgroup $SO(2m+1)$
since the standard Borel of $SO(2m+1)$ is a subgroup of the Borel
$B^{\bf z}$ of $\SpO(2|2m+1)$.

\begin{lemma}  \label{inequality}
Assume $\la =\sum_{i \in I(1|m)} \la_{i} \delta_i \in
X^\dag(T_{1,m})$. Retain the notations as above. Then, we have
\begin{align}
\la^{\bf z}_{-1} &\in \Z_+ \nonumber \\
 J(\la^{{\bf z}+}) &= \la^{+} \nonumber \\
 j (\la^{{\bf z}+})&=\la_{-1} -\la^{\bf z}_{-1} \nonumber \\
 j(\la^+) &\le \la_{-1}.  \label{eq3}
\end{align}
\end{lemma}

\begin{proof}
We have $\la^{\bf z}_{-1} \in \Z_+$ since
$2 \delta_{-1}$ is a positive root in $\Pi^{\bf z}$.

Via an ordered sequence of odd reflections associated to the $m$
odd roots
$$\delta_{m} -\delta_{-1}, \ldots, \delta_{1} -\delta_{-1},
$$
the weight $\la^{\bf z}$ relative to the Borel $B^{\bf z}$ is
replaced by the weight $\la$. The identity $J(\la^{{\bf z}+}) =
\la^{+}$ now follows by an argument which is completely parallel
to the one for Proposition~\ref{weight} (also compare the proof of
\cite[Lemma~6.2]{BKu}).

Next, we have
\begin{eqnarray*}
j (\la^{{\bf z}+}) &=&|\la^{{\bf z}+}| -|J(\la^{{\bf z}+})|  \\
 &=& |\la^{{\bf z}+}| -|\la^{+}| = \la_{-1}
-\la^{\bf z}_{-1}
\end{eqnarray*}
where the last equation is a byproduct of the above procedure of
odd reflections.

Finally, (\ref{eq3}) follows from a direct computation:
\begin{eqnarray*}
j(\la^+) &=&j(J(\la^{{\bf z}+})) \\ &\le& j (\la^{{\bf z}+})
= \la_{-1} -\la^{\bf z}_{-1} \le \la_{-1}.
\end{eqnarray*}
\end{proof}

The supergroup $SpO(2|2m+1)$ can be regarded naturally as a
subgroup of $SpO(2n|2m+1)$ in a way compatible with the natural
inclusion $I(1|m) \subseteq I(n|m)$, and thus the corresponding
fundamental systems of the two supergroups are compatible. Hence,
$\la =\sum_{i \in I(n|m)} \la_i \delta_i \in X^\dag(T)$ implies
that $\sum_{i \in I(1|m)} \la_i \delta_i \in X^\dag(T_{1,m})$, by
restricting the $SpO(2n|2m+1)$-module $L(\la)$ to the subgroup
$SpO(2|2m+1)$. It follows by (\ref{eq3}) that $j(\la^+) \le
\la_{-1}$ for $\la \in X^\dag(T)$. This together with
Lemma~\ref{easypart} imply that $X^\dag(T) \subseteq {\mathcal
X}^\dag(T)$ in the general case of $SpO(2n|2m+1)$. The proof of
Theorem~\ref{th:classif} for $SpO(2n|2m+1)$ is now completed.

\begin{remark}
Theorem~\ref{th:classif} for the classification of simple
$SpO(2n|2m)$-modules can be established using the same ideas in
two steps as above and so we will skip the details. The only
difference here is that the odd reflections associated to the odd
roots $\delta_i +\delta_j$ will also be used. The appearance of
the absolute value $|\la_{n+m}|$ comes from the weight condition
of $SO(2m)$ (compare Lemma~\ref{easypart} and its proof).
\end{remark}


\begin{thebibliography}{ABCD}

\bibitem{Bor} A. Borel, {\em Linear algebraic groups}, Springer Graduate Text in Math. {\bf 129}.

\bibitem{BK1} J. Brundan and A. Kleshchev,
{\em Projective representations of symmetric groups via Sergeev
duality}, Math. Z. {\bf 239} (2002), 27--68.

\bibitem{BK2} ---------,  {\em Modular
representations of the supergroup $Q(n)$, I},  J.~Algebra {\bf
260} (2003), 64--98.

\bibitem{BKu} J. Brundan and J. Kujawa, {\em A new proof of the Mullineux
conjecture}, J. Alg. Combin. {\bf 18} (2003), 13--39.

%\bibitem{CW} X. Cao and J. Wang, {\em Introduction to
%linear algebraic groups}, Science House, Beijing, 1987.

\bibitem{CPS} E. Cline, B. Parshall and L. Scott,
{\em On the tensor product theorem for algebraic groups}, J.
Algebra {\bf 63} (1980), 264--267.

%\bibitem{CPS} E. Cline, B. Parshall and L. Scott,
%{\em Cohomology, hyperalgebras and representations}, J. Algebra
%{\bf 63} (1980), 98--123.

%\bibitem{D} S. Donkin, {\em On Schur algebras and related algebras, I}, J. Algebra
%{\bf 104} (1986), 310--328.

\bibitem{FK} B. Ford and A. Kleshchev,
{\em A proof of the Mullineux conjecture}, Math. Z. {\bf 226}
(1997), 267--308.

\bibitem{Jac} N. Jacobson,  {\em Lie algebras},
Interscience, New York, 1962.

\bibitem{Jan} J. Jantzen, {\em Representations of algebraic groups}, Second Edition,
AMS, 2003.

\bibitem{Kac} V. Kac, {\em Lie superalgebras}, Adv. in Math. {\bf 26} (1977),
8--96.

%\bibitem{Kl} A. Kleshchev,
%{\em Branching rules for modular representations of symmetric
%groups III}, J. London Math. Soc. {\bf 54} (1996), 25--38.

\bibitem{Ku} J. Kujawa, {\em The Steinberg tensor product theorem for
$GL(m|n)$},  Representations of algebraic groups, quantum groups,
and Lie algebras, Contemp. Math., {\bf 413} (2006), 123--132.

\bibitem{LSS} D. Leites, M. Saveliev and V. Serganova,
{\em Embedding of ${\rm osp}(N/2)$ and the associated nonlinear
supersymmetric equations}. Group theoretical methods in physics,
vol. I (Yurmala 1985), 255--297, VNU Sci. Press, Utrecht, 1986.

\bibitem{Ma} Yu. Manin, {\em Gauge field theory and complex
geometry}, Grundlehren der mathematischen  Wissenschaften {\bf
289}, Second Edition, Springer, 1997.

\bibitem{Mu} G. Mullineux, {\em Bijections of $p$-regular
partitions and $p$-modular irreducibles of symmetric groups}, J.
London Math. Soc. {\bf 20} (1979), 60--66.

\bibitem{PS} I. Penkov and V. Serganova,
{\em Generic irreducible representations of finite-dimensional Lie
superalgebras}, Internat. J. of Math. {\bf 5} (1994), 389--419.

\bibitem{Ser} V. Serganova,
{\em Characters of Irreducible Representations of Simple Lie
Superalgebras}, Doc. Math. J. DMV Extra Volume ICM {\bf II}
(1998), 583--593.

\bibitem{St1} R. Steinberg, {\em Representations of algebraic groups},
Nagoya Math. J. {\bf 22} (1963), 33--56.

\bibitem{St2} ---------, {\em Lectures on Chevalley groups},
Mimeographic Lecture notes, Yale University, 1968.

\bibitem{Xu} M. Xu, {\em On Mullineux' conjecture in the
representation theory of symmetric groups}, Commun. Alg. {\bf 25}
(1997), 1797--1803.

\end{thebibliography}
\end{document}